\documentclass[12pt]{amsart}
\usepackage{amsmath, amsthm, amssymb, hyperref, color}
\usepackage[linesnumbered,commentsnumbered,ruled,vlined]{algorithm2e}
\usepackage{graphicx}
\usepackage{caption}
\usepackage{subcaption}
\usepackage{mathtools}
\usepackage{enumerate}
\usepackage{verbatim}
\usepackage{upgreek}
\usepackage{tikz}
\usepackage{caption}
\usepackage{subcaption}
\usepackage{booktabs}
\usepackage{colortbl}
\usepackage{xcolor}

\newcolumntype{L}{>{$}l<{$}}

\usepackage{fullpage}
\usepackage{xspace}

\newtheorem{theorem}{Theorem}

\newtheorem{proposition}[theorem]{Proposition}
\newtheorem{lemma}[theorem]{Lemma}

\theoremstyle{definition}
\newtheorem{definition}[theorem]{Definition}
\newtheorem{problem}[theorem]{Problem}
\newtheorem{remark}[theorem]{Remark}
\newtheorem{conjecture}[theorem]{Conjecture}
\newtheorem{cor}[theorem]{Corollary}

\newtheorem{examplex}[theorem]{Example}
\newenvironment{example}
{\pushQED{\qed}\examplex}
{\popQED\endexamplex}

\numberwithin{theorem}{section}

\newcommand{\QQ}{\mathbb{Q}}
\newcommand{\CC}{\mathbb{C} }

\newcommand{\NN}{\mathbb{N}}

\newcommand{\A}{\mathcal{A}}
\newcommand{\Shi}{\widehat{\mathcal{A}}_n}

\definecolor{darkgreen}{rgb}{0.0,0.1,0.6}
\newcommand{\df}[1]{{\bf\color{darkgreen} #1}}

\newcommand\csa{connected subgraph arrangement\xspace}
\newcommand\csas{connected subgraph arrangements\xspace}
\newcommand\MAT{$\operatorname{MAT}$\xspace}

\DeclareMathOperator{\rk}{rk}
\DeclareMathOperator{\codim}{codim}
\DeclareMathOperator{\Der}{Der}

\usepackage{mathdots}

\title[On arrangements of hyperplanes from connected subgraphs]
{On arrangements of hyperplanes\\ from connected subgraphs}

\author{Michael~Cuntz}
\address{Michael Cuntz, Leibniz Universit\"at Hannover,
	Institut f\"ur Algebra, Zah\-lentheorie und Diskrete Mathematik,
	Fakult\"at f\"ur Mathematik und Physik,
	Wel\-fengarten 1,
	D-30167 Hannover, Germany}
\email{cuntz@math.uni-hannover.de}

\author{Lukas~K\"uhne}
\address{Lukas K\"uhne, Fakult\"at f\"ur Mathematik, Universit\"at Bielefeld, Bielefeld, Germany}
\email{lkuehne@math.uni-bielefeld.de}


\begin{document}
	
	\maketitle
	
	\begin{abstract}
We investigate arrangements of hyperplanes whose normal vectors are given by connected subgraphs of a fixed graph. These include the resonance arrangement and certain ideal subarrangements of Weyl arrangements. We characterize those which are free, simplicial, factored, or supersolvable. In particular, such an arrangement is free if and only if the graph is a cycle, a path, an almost path, or a path with a triangle attached to it.
	\end{abstract}

	\section{Introduction} \label{section1}
	
	Let $n\in\NN$, $F$ be a field, and $V:=F^n$. Denote by $e_1,\ldots,e_n$ the standard basis of $V$ and by $x_1,\ldots,x_n$ its dual basis in $V^*$.
	We now introduce the main protagonist of this article which to the best of our knowledge has not appeared in the literature yet.
	
	\begin{definition}\label{def:csa}
		Let $G=(N,E)$ be an undirected (simple) graph with vertex set $N=\{1,\dots,n\}$ and edges $E$.
		We define the \df{\csa} $\A_G$ in $F^n$ as
		\[
		\A_G(F) \coloneqq \{ H_I\mid \emptyset \neq I\subseteq N \quad \text{if } G\left[I\right] \text{ is connected}\},
		\]
		where $H_I$ is the hyperplane
		\[ H_I=\ker \sum_{i\in I}x_i \]
		and $G\left[I\right] $ is the induced subgraph on the vertices $I\subseteq N$.
		We define $\A_G\coloneqq \A_G(\QQ)$ for short.
	\end{definition}
	
	This definition includes a series of well-known arrangements:
	
	\begin{example}\label{typeA}
		If $F=\QQ$ and $G_n$ is the Dynkin diagram of type $A_n$, i.e.\ the path graph on $n$ vertices, then the arrangement $\A_{G_n}$ is the reflection arrangement of type $A_n$ also called \df{braid arrangement}.
		The change of coordinates $(\QQ^n)^*\to (\QQ^{n+1})^*$, $x_i\mapsto x_i -x_{i+1}$ maps the normal vectors of $\A_{G_n}$ to the braid normal vectors $x_i-x_j$ for $1\le i < j \le n+1$.
		
		In this case, the characteristic polynomial (see \ref{def:char_pol}) of $\A_{G_n}$ is $\chi(\A_{G_n},t)=\prod_{i=1}^n (t-i)$ and the arrangement has $n!$ many chambers.
	\end{example}
	
	\begin{example}
		If $F=\QQ$ and $K_n$ is a complete graph on $n$ vertices, then the arrangement $\A_{K_n}$ is the \df{resonance} or \df{all-subsets arrangement}, that is, it consists of the $2^n-1$ hyperplanes having all nonzero $0/1$-vectors in $\QQ^n$ as normal vectors.
		
		This arrangement is relevant for applications in mathematical physics and economy, see for instance~\cite{kuehne_resonance}.
		Computing its number of chambers or its characteristic polynomial is computationally difficult, both are only known up to $n\le 9$~\cite{BEK_resonance,CS_resonance}.
	\end{example}
	
	Hence, the \csas interpolate between these two families of arrangements depending on the underlying graph. But there are more examples:
	
	\begin{example}
		Let $F=\QQ$ and $G$ be the Coxeter diagram of a finite Weyl group (or more generally of a finite Weyl groupoid) $W$. Then it is known (see \cite[Prop.\ 3.2]{p-CH10}) that $\Gamma$ is a connected subgraph of $G$ if and only if the sum of the simple roots labeled by the vertices of $\Gamma$ is a positive root of $W$:
		\[ \Gamma \text{ connected subgraph } G
\quad \Longleftrightarrow \quad \sum_{i \text{ vertex of } \Gamma} x_i \quad \text{is a positive root of } W.\]
    Since the defining normals of $\A_G$ are $0/1$-vectors, it follows that $\A_G$ is an \df{ideal subarrangement} given by an ideal of the root poset of $W$. In particular in the case of a Weyl group, combinatorics and freeness of these arrangements have already been studied (\cite{ABCHT_mat}, \cite{CRS17}).\\
		Conversely, any \csa may be viewed as an ideal subarrangement when the Weyl group is allowed to be infinite. However, not much is known about the root poset structure of infinite root systems.
	\end{example}

	In this article we want to investigate \csas from different perspectives:
	\begin{enumerate}
	    \item \df{Free arrangements} (see Def.~\ref{def:free}): An arrangement is free if a certain module of logarithmic derivations is free. The study of free arrangements was pioneered by Terao~\cite{Terao_addition} and is by now one of the most studied properties of arrangements~\cite{Yos14}.
	    \item \df{Simplicial arrangements} (see Def.~\ref{def:simp}): Simplicial arrangements are arrangements that cut the space into simplicial cones. For instance a real reflection arrangement is simplicial since all chambers are simplicial cones. It is known that the complexified complement of a simplicial arrangement is a $K(\pi,1)$-space \cite{Del72}. In rank three, there is a catalogue of known simplicial arrangements \cite{p-G-09}, extended later by further arrangements in \cite{Cun12} and \cite{Cun22}; it is an open question whether this catalogue is complete.
	    \item \df{Factored and supersolvable arrangements} (see Def.~\ref{def:factored}, Def.~\ref{def:ss}): Factored (also known as nice partition) arrangememts and supersolvable arrangements are combinatorially defined and motivated by a decomposition of the Orlik-Solomon algebra~\cite{Ter92} and the subgroup lattice of a finite supersolvable group~\cite{Sta72}, respectively.
	    Both notions have been thoroughly researched over the times, for some of the latest contributions see for instance~\cite{CM19,AD20,MMR22}.
	    \item \df{Aspherical or $K(\pi,1)$-arrangements}: A complex arrangement $\A$ is aspherical or $K(\pi,1)$ if the universal covering space of the complement of $\A$ is contractible with a fundamental group isomorphic to $\pi$.
	    Deligne proved that complexified simplicial arrangements are $K(\pi,1)$~\cite{Del72}.
	\end{enumerate}
\medskip
The following four families of graphs will appear in the characterization of the above types within the class of \csas.
	\begin{enumerate}
		\item \df{Path graph} $P_n$.
		As discussed above in Example~\ref{typeA} the \csa is the braid arrangement in this case.
		\item \df{Cycle graph} $C_n$.
		In this case the the characteristic polynomial of the \csa $\A_{C_n}$ is
		\[
		\chi(\A_{C_n},t)=(t-1)(t-n)^{n-1}.
		\]
		In Section~\ref{sec:cycle} we show that $\A_{C_n}$ is equal to the cone of a \df{Shi arrangement} after a change of coordinates~\cite{Shi}.
		Following work of Athanasiadis~\cite{Athanasiadis}, the latter is known to be inductively free with exponents equal to the roots of $\chi(\A_{C_n},t)$.
		\item \df{Almost path graph} $A_{n,k}$ for $1<k<n$: This is a graph on the vertices $\left[n+1\right]$ consisting of a path between the vertices $1$ to $n$ and an additional edge between the vertices $k$ and $n+1$.
		In Section~\ref{sec:almost_path} we show that $\A_{A_{n,k}}$ is \MAT-free with exponents
		$$(1,k+1,\ldots,n+1,n-k+2,\ldots,n).$$
		Its characteristic polynomial therefore has only these integral roots.
		\item \df{Path-with-triangle-graph} $\Delta_{n,k}$ for $1<k<n$: This is also a graph on the vertices $\left[n+1\right]$ consisting of a path between the vertices $1$ to $n$ and two additional edges:  one between the vertices $k$ and $n+1$ and one between the vertices $k+1$ and $n+1$.
		In~Section~\ref{sec:triangle_path} we show that 	$\A_{\Delta_{n,k}}$ is free.
	\end{enumerate}

\begin{figure}[ht]
\centering
\includegraphics[width=.4\linewidth]{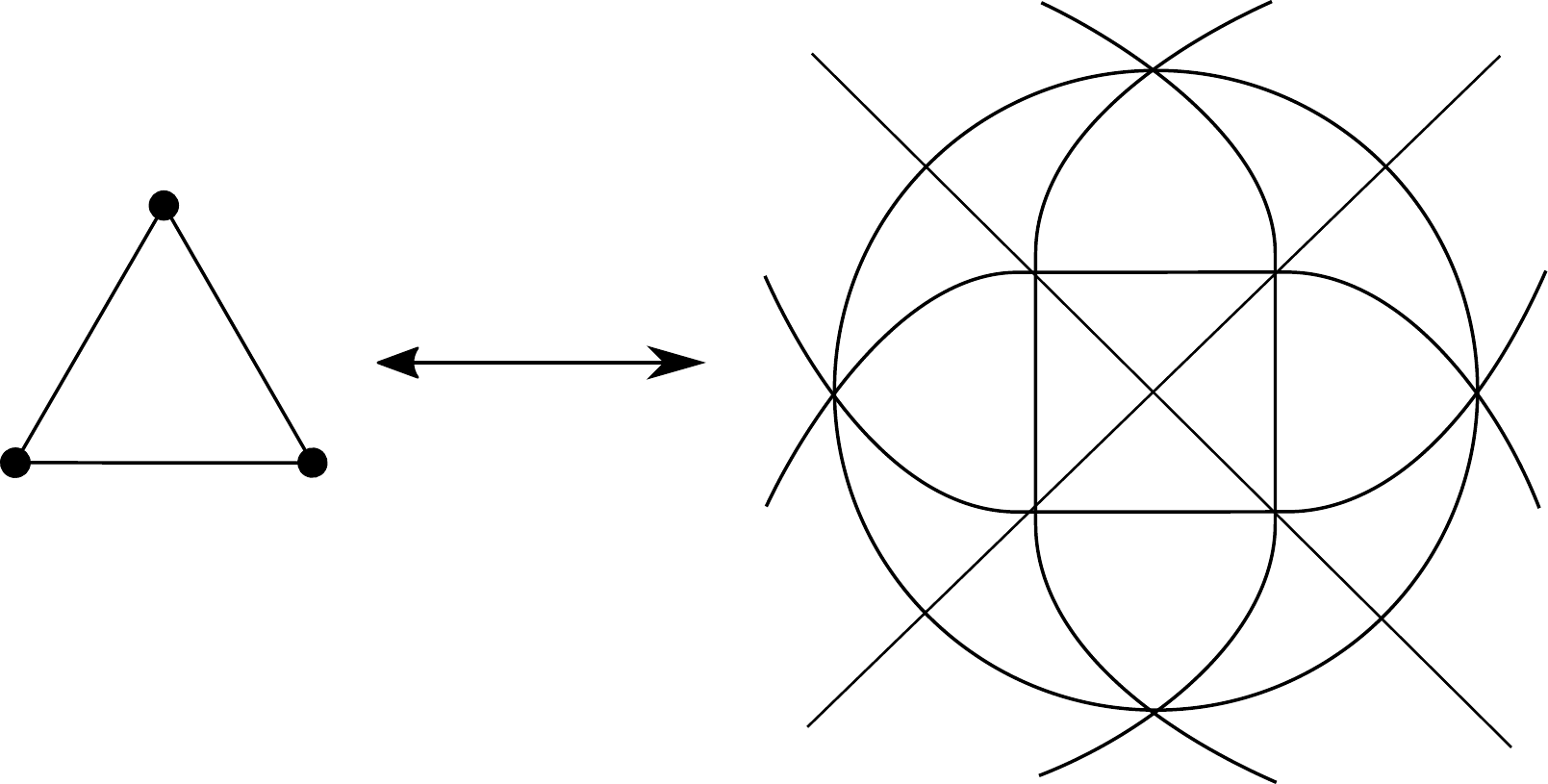}
\caption{The \csa $\A_{C_3}=\A_{\Delta_{2,1}}$.\label{fig:AC3}}
\end{figure}
	\begin{example}
If $n=3$, $F=\QQ$ and
\[ R := \{(1,0,0),(0,1,0),(0,0,1),(1,1,0),(0,1,1),(1,0,1),(1,1,1)\}\subset (\QQ^3)^*. \]
Then $\A_{C_3}=\A_{\Delta_{2,1}}=\{\ker \alpha\mid \alpha \in R\}$ is (up to linear transformations) the simplicial arrangement with $7$ hyperplanes in $\QQ^3$ (see Figure \ref{fig:AC3}).
	\end{example}

	Our main theorem classifies the free \csas over $\QQ$:
	
	\begin{theorem}[Prop.~\ref{prop:cycle}, Thm.~\ref{thm:mat_almost_path}, Thm.~\ref{thm:path_triangle}, Thm.~\ref{thm:classification}]\label{thm:main}
		Let $G$ be a connected graph.
		The \csa $\A_G$ is free if and only if $G$ is a member of one of the four families described above, that is a path, a cycle, an almost path, or a path-with-triangle graph.
	\end{theorem}
	For a disconnected graph the above theorem can be applied to its connected components as this corresponds to direct products for arrangements.
	Moreover, we can classify the simplicial, factored, and supersolvable \csas.
	\begin{theorem}[Thm.~\ref{thm:ss_p}]\label{thm:sim}
		Let $G$ be a connected graph.
        The \csa $\A_G$ is simplicial if and only if $G$ is a triangle or a path graph.
	\end{theorem}

	\begin{theorem}[Thm.~\ref{thm:factored}, Cor.~\ref{cor:ss}]\label{thm:fac}
		Let $G$ be a connected graph.
		\begin{enumerate}
			\item The \csa $A_G$ is factored if and only if $G$ is a path graph or a path-with-triangle graph $\Delta_{n,1}$ for $n\ge 2$.
			\item 	The \csa $\A_G$ is supersolvable if and only if $G$ is a path.
		\end{enumerate}
	\end{theorem}

	With regard to $K(\pi,1)$-arrangements we would like to pose the following question:
	\begin{problem}
		Which \csas are $K(\pi,1)$?
	\end{problem}
	A factored arrangement in $\CC^3$ is $K(\pi,1)$ due to Paris~\cite{Par95}.
	Falk Randell asked whether a general complex factored arrangement is $K(\pi,1)$~\cite{FR00}.
	Results in~\cite{MMR22} provide evidence for this question.
	We therefore conjecture that the \csas of path graphs and the graphs $\Delta_{n,1}$ for $n\ge 2$ are $K(\pi,1)$.
	
	Besides the classification of free \csas, we have no characterization of graphs $G$ such that the characteristic polynomial $\chi(\A_G)$ has only integral roots yet. We conjecture:
    \begin{conjecture}
    If $G$ is a connected graph, then
    $\chi(\A_G)$ has only integral roots if and only if $\A_G$ is free.
    \end{conjecture}

	This article is organized as follows.
	We start by discussing the relevant background on hyperplane arrangements in Section~\ref{sec:pre}.
	Subsequently we prove that the \csas of cycles, almost paths and path-with-triangle-graphs are free in Sections~\ref{sec:cycle},~\ref{sec:almost_path}, and~\ref{sec:triangle_path}, respectively.
	We finish the proof of Theorem~\ref{thm:main} in Section~\ref{sec:converse} by proving that the \csas of other graphs are not free using localizations of arrangements.
	We conclude by discussing simplicial, factored and supersolvable arrangements in Sections~\ref{sec:simplicial} and~\ref{sec:factored}.
	
	\section{Preliminaries}\label{sec:pre}
	
	\subsection{The characteristic polynomial}
	
	\begin{definition}
		Let $F$ be a field, $n\in\NN$, and $V:=F^n$. An \df{arrangement of hyperplanes} $(\A,V)$ (or $\A$ for short) is a finite set of hyperplanes $\A$ in $V$.
	\end{definition}
	
	\begin{definition}[{\cite[Def.\ 1.12]{OT92}}]
		Let $(\A,V)$ be an arrangement.
		The \df{intersection lattice} $L(\A)$ of $\A$ consists of all intersections of elements of $\A$
		including $V$ as the empty intersection. 
		The \df{rank} $\rk(\A)$ of $\A$ is defined as the codimension of the intersection of all hyperplanes in $\A$.
		For $0 \leq k \leq n$ we write $L_k(\A) := \{ X \in L(\A) \mid \codim(X) = k \}$.
	\end{definition}
	
	\begin{definition}[{\cite[Def.\ 1.13]{OT92}}]
		Let $(\A,V)$ be an arrangement of hyperplanes.
		For $X\in L(\A)$, we define the \df{localization}
		$$
		\A_X := \{ H \in \A \mid X \subseteq H \}
		$$ 
		of $\A$ at $X$, and the \df{restriction} $(\A^X,X)$ of $\A$ to $X$, where 
		$$
		\A^X := \{ X\cap H \mid H \in \A \setminus \A_X \text{ and } X\cap H\ne \emptyset \}.
		$$
	\end{definition}
	
	\begin{definition}[{\cite[Cor.\ 2.57]{OT92}}]\label{def:char_pol}
		Let $(\A,V)$ be an arrangement of hyperplanes.
		The \df{characteristic polynomial} $\chi(\A)\in \CC[t]$ of $\A$ is defined recursively:
		\begin{itemize}
			\item If $\A=\emptyset$, then $\chi(\A)=t^{\rk(\A)}$.
			\item If $H\in \A$, then $\chi(\A)=\chi(\A\backslash \{H\})-\chi(\A^H)$.
		\end{itemize}
(Note that the second case is independent of the choice of $H$.)
	\end{definition}
	
	\subsection{Freeness}
	
	Let $S = S(V^*)$ be the symmetric algebra of the dual space and let $\Der(S)$ be the $S$-module of $F$-derivations of $S$.
	\begin{definition}[{\cite[Def.\ 1.20]{OT92}}]\label{def:free}
		Let $\A$ be an arrangement of hyperplanes in $V$.
		For $H \in \A$ we fix $\alpha_H \in V^*$ with $H = \ker(\alpha_H)$.
		The \df{module of $\A$-derivations} is defined by
		\begin{equation*}
		D(\A) := \{ \theta \in \Der(S) \mid \theta(\alpha_H) \in {\alpha_H}S \text{ for all } H \in \A\}.
		\end{equation*}
		We say that $\A$ is \df{free} if the module of $\A$-derivations is a free $S$-module. The \df{exponents} of $\A$ are the degrees of the elements of a basis of $D(\A)$.
	\end{definition}
	
	\begin{theorem}[\df{Terao's Factorization Theorem}, {\cite[4.137]{OT92}}]\label{thm:factorization}
		If $\A$ is free with exponents $(d_1,...,d_\ell)$, then $$\chi(\A)=\prod_{i=1}^\ell (t-d_i).$$
	\end{theorem}

Terao's addition-deletion Theorem is a useful tool to compute exponents and freeness for large classes of arrangements:

	\begin{theorem}[\cite{Terao_addition}]\label{thm:addition_deletion}
		Let $\A$ be a nonempty arrangement in an $\ell$-dimensional vector space $V$, $H_0\in\A$.
		Then any of the following two statements imply the third:
		\begin{enumerate}[(1)]
			\item $\A$ is free with exponents $(d_1,...,d_\ell)$.
			\item $\A\setminus\{H_0\}$ is free with exponents $(d_1,...,d_\ell-1)$.
			\item $\A^{H_0}$ is free with exponents $(d_1,...,d_{\ell-1})$.
		\end{enumerate}
	\end{theorem}

A result in the same spirit is the Multiple-Addition-Theorem; it was first used to prove that ideal subarrangements of Weyl arrangements are free:

	\begin{theorem}[\df{MAT}, {\cite[Thm.~3.1]{ABCHT_mat}}]\label{thm:MAT}
		Let $\A'$ be a free arrangement with $\exp(\A') = (d_1,...,d_\ell)$, $(d_1 \le \dots \le d_\ell)$, and $1 \le p \le \ell$ the multiplicity of the highest exponent, i.e.,
		\[
		d_1 \le \dots \le d_{\ell - p} < d_{\ell-p+1} = \dots = d_\ell =d.
		\]
		Let $H_1,\dots,H_q$ be hyperplanes with $H_i \notin \A'$ for $i = 1,\dots,q$.
		Define
		\[
		\A_j'' \coloneqq (A'\cup {H_j})^{H_j} = \{H \cap H_j \mid H\in\A' \} \, (j=1,\dots,q).
		\]
		Assume that the following three conditions are satisfied:
		\begin{enumerate}
			\item $X\coloneqq H_1\cap \dots \cap H_q$ is $q$-codimensional.
			\item $X\not \subseteq  \bigcup_{H\in \A'}H$.
			\item\label{it:tripel} $|\A'|-|\A_j''|=d$ for all $1\le j\le q$.
		\end{enumerate}
		Then $q\le p$ and $\A\coloneqq \A'\cup \{H_1,\dots,H_q\}$ is free with $\exp(\A)=(d_1,\dots,d_{\ell-q},(d+1)^q)$.
	\end{theorem}

This motivates the following definition (see also Example \ref{example:MAT}):
	
	\begin{definition}[{\cite{CM_mat}}]
		An arrangement $\A$ is \df{\MAT-free} if there exists a chain of arrangements
		\[
		\A_0\subsetneq \A_1\subsetneq \dots\subsetneq\A_k=\A
		\]
		such that $\A_0$ is the empty arrangement and every pair of arrangements $\A_i$ and $\A_{i+1}$ for $0\le i < k$ satisfies the assumptions of Theorem~\ref{thm:MAT}.
	\end{definition}

	\subsection{0/1-Arrangements}
	We call an arrangement in $\QQ^n$ a \df{$0/1$-arrangement} if all its hyperplanes are of the form $H_I$ for some subsets $I\subseteq \left[n\right]$, where $H_I$ is the hyperplane $\ker \sum_{i\in I} x_i$ as defined in Definition~\ref{def:csa}.
	We will often argue via these subsets in the following.	
	Note that \csas and their subarrangements are all $0/1$-arrangements by definition.
	One of the aspects why $0/1$-arrangements are easier to study than general arrangements is the following lemma.
	
	\begin{lemma}\label{lem:triples}
		For pairwise different nonempty subsets $A_1,A_2,A_3\subseteq \left[n\right]$ and $n\ge 1$ the following two conditions are equivalent:
		\begin{enumerate}
			\item[$(i)$] $\rk (H_{A_1}\cap H_{A_2}\cap H_{A_3} )= 2$ over $\QQ$ and
			\item[$(ii)$] $A_{i_1}\dot{\cup} A_{i_2}=A_{i_3}$ for some choice of pairwise different indices $1\le i_1,i_2,i_3\le 3$ where $\dot{\cup}$ denotes a disjoint set union.
		\end{enumerate}
		If these conditions are satisfied we call $\{A_1,A_2,A_3\}$ a \df{circuit-triple}.
	\end{lemma}
	
	\begin{proof}
		Denote by $\chi_{A_i}$ the characteristic vector of the set $A_i$ in $\QQ^n$, which is a $0/1$-vector with entry $1$ in coordinate $j$ if $j\in A_i$.
		Assume condition $(i)$ holds.
		Thus, the matrix with columns $\chi_{A_1},\chi_{A_2},\chi_{A_3}$ has rank $2$ and there exist $\lambda_1,\lambda_2,\lambda_3\in \QQ$ not all zero with $\sum_{i=1}^3\lambda_iA_i=0$.
		As the sets $A_1,A_2,A_3$ are pairwise different we must have $\lambda_i\neq 0$ for all $1\le i\le 3$.
		Therefore we can without loss of generality assume that $\lambda_1,\lambda_2>0$ and $\lambda_3<0$.
		
		For $1\le j \le n$ set $s_j\coloneqq \sum_{i=1}^3 |A_i\cap \{j\}|$, that is, $s_j$ is the number of times $j$ appears in $A_1,A_2,A_3$.
		The relation $\sum_{i=1}^3\lambda_iA_i=0$ with all $\lambda_i$ being nonzero implies that $s_j$ cannot be $1$ for any $1\le j\le n$.
		
		Assume $s_j=3$ for some $1\le j\le n$.
		The relation  $\sum_{i=1}^3\lambda_iA_i=0$ then yields $\lambda_1+\lambda_2=-\lambda_3$ as we assumed $\lambda_1,\lambda_2>0$ and $\lambda_3<0$.
		Therefore, $s_k$ cannot be $2$ for any $1\le k\le n$ and $s_k$ is either $0$ or $3$ for all $1\le k\le n$.
		This implies that the sets $A_1,A_2,A_3$ are all equal contrary to our assumption.
		
		Hence, $s_k$ is either $0$ or $2$ for all $1\le k\le n$.
		This yields $\lambda_1=-\lambda_3$ and $\lambda_2=-\lambda_3$ using the relation $\sum_{i=1}^3\lambda_iA_i=0$ and our assumption on the signs of $\lambda_i$.
		Therefore, we have $A_1\cap A_2=\emptyset$ and $A_1\cup A_2=A_3$. Thus, condition $(ii)$ holds.
		
		For the converse, observe that $A_{i_1}\dot{\cup} A_{i_2}=A_{i_3}$ implies $\chi_{A_{i_1}}+\chi_{A_{i_2}}=\chi_{A_{i_3}}$ which immediately yields condition $(i)$.
	\end{proof}
	
	Using this lemma we can reduce Condition~\eqref{it:tripel} in Theorem~\ref{thm:MAT} to the following simpler set-theoretic condition for $0/1$-arrangements.
	\begin{lemma}
		Let $\A'=\{H_{I_1},\dots,H_{I_m}\}$ be a $0/1$ arrangement in $\QQ^n$ and $H_I$ an additional hyperplane with $H_I\notin \A'$ and $I_1,\dots,I_m,I\subseteq \left[n\right]$.
		Set $\A_I'' \coloneqq (\A'\cup H_I)^{H_I}$ and $\mathcal{I}\coloneqq \{I_1,\dots,I_m\}$.
		Then
		\begin{align*}
		|\A'|-|\A_I''| = |\{\mbox{circuit-triples among the sets in }\mathcal{I}\cup \{I\}\mbox{ that involve }I\}|.
		\end{align*}
	\end{lemma}
	\begin{proof}
		By definition we have
		\begin{align*}
		|\A_I''| =|\{H_{I_1}\cap H_I,\dots,H_{I_m}\cap H_I\}|.
		\end{align*}
		By Lemma~\ref{lem:triples} pairs of elements in this set agree if and only if the two corresponding subsets form a circuit-triple with $I$.
		All others sets are distinct and their contribution cancels in the difference $|\A'|-|\A_I''|$ with the corresponding hyperplanes appearing in $\A'$ .
	\end{proof}

	\section{Cycle graphs}\label{sec:cycle}
	
	The $n$-th \df{Shi arrangement} $\Shi$ (of type $A$) consists of the affine hyperplanes
	\[
	\ker (x_i-x_j),\:\: \ker (x_i-x_j-1) \mbox{ for }1\le i<j\le n.
	\]
	Therefore the cone $c\Shi$ of the $n$-th Shi arrangement is an arrangement in $\QQ^{n+1}$ and consists of the hyperplanes
	\[
	\begin{array}{lll}
	\ker & x_i-x_j & \text{for }1\le i<j\le n,\\
	\ker & x_i-x_j-x_{n+1} & \text{for }1\le i<j\le n,\\
	\ker & x_{n+1}.\\
	\end{array}
	\]
	Athanasiadis showed that $c\Shi$ is inductively free with $\exp(c\Shi)=(0,1,n,\dots,n)$\cite{Athanasiadis}.
	
	\begin{proposition}\label{prop:cycle}
		Let $F=\QQ$ and $C_n$ be the cycle graph on $n$ vertices.
		Then there exists a linear embedding $\phi:(\QQ^{n})^* \to (\QQ^{n+1})^*$ that maps the \csa $\A_{C_n}$ to the cone of the Shi arrangement $c\Shi$.
		Therefore, the \csa $\A_{C_n}$ is inductively free with exponents $\exp(\Shi)=(1,n,\dots,n)$.
		It has the characteristic polynomial $\chi(\A_{C_n},t)=(t-1)(t-n)^{n-1}$.
		Hence, by Zaslavsky's result $\A_{C_n}$ has $(-1)^n\chi(\A_{C_n},-1)=2 (n+1)^{n-1}$ chambers.
	\end{proposition}
	\begin{proof}
		Consider the map
		\[
		\phi:(\QQ^n)^* \to (\QQ^{n+1})^*,\quad
		x_i \mapsto 
		\begin{cases}
		x_{i+1}-x_i & \mbox{ if }i <n,\\
		x_1-x_n-x_{n+1} &\mbox{ if }i =n.
		\end{cases}
		\]
		The map $\phi$ is a linear embedding of $(\QQ^n)^*$ into $(\QQ^{n+1})^*$.
		We claim that $\phi$ maps the hyperplanes of the \csa $\A_{C_n}$ bijectively to the ones of $c\Shi$.
		Recall that by definition the hyperplanes of $\A_{C_n}$ are given by
		\[ \begin{array}{lll}
		\ker & \sum_{k=i}^j x_k & \mbox{ for } 1\le i \le j \le n,\\
		\ker & \sum_{k=j+1}^n x_k+\sum_{k=1}^i x_k  & \mbox{ for } 1\le i < j \le n-1.
		\end{array}
		\]
		We therefore describe the effect of $\phi$ for each of these hyperplanes:
		\begin{description}
			\item[$\ker \sum_{k=i}^j x_k \mbox{ for } 1\le i \le j < n$] The map $\phi$ sends this hyperplane to $\ker x_{j+1}-x_i$ which is a hyperplane of $c\Shi$.
			\item[$\ker \sum_{k=i}^n x_k \mbox{ for } 1 \le i \le n$] The map $\phi$ sends this hyperplane to $\ker x_1-x_i-x_{n+1}$ which is a hyperplane of $c\Shi$. Note that the case $i=1$ yields the hyperplane $\ker x_{n+1}$ of $c\Shi$.
			\item[$\ker \sum_{k=j+1}^n x_k+\sum_{k=1}^i x_k \mbox{ for } 1\le i < j \le n-1$] The map $\phi$ sends this hyperplane to $\ker x_{i+1}-x_{j+1}-x_{n+1}$ which is also a hyperplane of $c\Shi$.
		\end{description}
		Since every hyperplane in $c\Shi$ is in the image of $\A_{C_n}$ under the map $\phi$ and the number of hyperplanes in both arrangements is equal this defines a bijection between the hyperplanes of these arrangements as claimed.
	\end{proof}

	\section{Almost path graphs}\label{sec:almost_path}
	In this section we consider the \csa associated with the \df{almost-path graphs} $A_{n,k}$ for $1<k<n$ which are graphs on the vertices $\left[n+1\right]$ consisting of a path from $1$ to $n$ together with an additional edge connecting $k$ and $n+1$.
	For instance, Figure~\ref{fig:almost_path} depicts the graph $A_{7,4}$ in this family.
	
	\begin{figure}[hb]
		\begin{tikzpicture}[scale=0.5,mynode/.style={font=\color{#1}\sffamily,circle,draw=black,inner sep=1pt,minimum size=0.5cm}]
		\node[mynode=black] (v1) at (0,0) {$1$};
		\node[mynode=black] (v2) at (3,0) {$2$};
		\node[mynode=black] (v3) at (6,0) {$3$};
		\node[mynode=black] (v4) at (9,0) {$4$};
		\node[mynode=black] (v5) at (12,0) {$5$};
		\node[mynode=black] (v6) at (15,0) {$6$};
		\node[mynode=black] (v7) at (18,0) {$7$};
		
		\node[mynode=black] (v8) at (9,-2) {$8$};
		
		\draw[gray,very thick] (v1)--(v2)--(v3)--(v4)--(v5)--(v6) -- (v7)(v8) --(v4);
		\end{tikzpicture}
		\caption{The almost-path graph $A_{7,4}$.}
		\label{fig:almost_path}
	\end{figure}
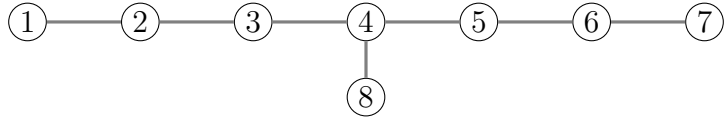
	
	\begin{theorem}\label{thm:mat_almost_path}
		Let $A_{n,k}$, $1<k<n$, be an almost path graph on $n+1$ vertices, that is, it consists of a path from vertex $1$ to vertex $n$ with an additional edge connecting vertex $k$ and vertex $n+1$.
		Then the \csa $\A_{A_{n,k}}$ is MAT free with exponents $(1,k+1,\dots,n+1,n-k+2,\dots,n)$.
	\end{theorem}
	\begin{proof}
		We consider the following chain of arrangements:
		\begin{align*}
		\A_0 &: \mbox{ the empty arrangement in } \QQ^{n+1},\\
		\A_{i} &\coloneqq \A_{i-1} \cup \{ H_{I} \mid H_I\in \A_{A_{n,k}} \mbox{ and }|I|=i\} \mbox{ for }1\le i \le n+1.
		\end{align*}
		
		By construction we have $\A_{n+1}=\A_{A_{n,k}}$, $|\A_1| = n+1$.
		Thus after proving that each pair $(\A_{i-1}, \A_{i})$ satisfies the conditions in Theorem~\ref{thm:MAT}, we can conclude that $\A_{n+1}$ is MAT free with exponents $(d_1,\dots,d_{n+1})$ as defined above.
		
		For the pair $(\A_0, \A_{1})$ all conditions are trivially satisfied as the intersection of the $n+1$ hyperplanes in $\A_1$ is zero-dimensional (the coefficient vectors are linearly independent) and each hyperplane is not involved in any circuit-triple with $\A_0$.
		So consider a pair $(\A_{i-1},\A_i)$ for a fixed $2\le i\le n+1$.
		
		The first step is to show that the intersection of the hyperplanes in $\A_{i}\setminus \A_{i-1}$ are of maximal codimension, or equivalently, the family of vectors defining these hyperplanes is linearly independent.
		For convenience, we set the vector $v_I\in \QQ^{n+1}$ to be the defining vector of the hyperplane $H_I$ for all $I\subseteq \left[n+1\right]$.
		Assume there are $\lambda_I$ for $H_I\in \A_{i}\setminus \A_{i-1}$ such that 
		\begin{align}\label{eq:dependence}
		\sum_{H_I\in \A_{i}\setminus \A_{i-1} }\lambda_I v_I = 0.
		\end{align}
		Let's give names to the hyperplanes in 
		$\A_{i}\setminus \A_{i-1}$ which we use in the following arguments:
		We denote a hyperplane $H_I\in \A_{i}\setminus \A_{i-1}$ by $H_s^+$ if both $s=\min I$ and $n+1\in I$, and by $H_s^-$ if both $s=\min I$ and $n+1\not \in I$.
		Further denote the corresponding $\lambda$-coefficient from Equation~\eqref{eq:dependence} by $\lambda^+_s$ and $\lambda^-_s$.
		Let $s_0$ and $s_1$ be the first and the last index such that both hyperplanes $H^+_{s}$ and $H^-_{s}$ exist in $\A_{i}\setminus \A_{i-1}$.

		By considering the hyperplanes $H^-_1,H^-_2,\dots,H_{s_0-1}^-$ (whose sets all don't contain $n+1$) in this order in light of Equation~\eqref{eq:dependence} we can conclude that $\lambda_s^-=0$ for all $1\le s<s_0$.
		Analogously considering the hyperplanes $H^-_{s_1+1},H^-_{s_1+2},\dots$ in the reverse order in light of Equation~\eqref{eq:dependence} we can conclude that $\lambda_s^-=0$ for all $s_1<s$.
		Secondly, consider the pairs of hyperplanes $H_s^+$ and $H_s^-$ for $s_0\le s\le s_1$.
		Equation~\eqref{eq:dependence} readily implies that $\lambda_s^+=-\lambda_s^-$ for all $s_0\le s \le s_1$.
		By considering the hyperplanes $H_{s-1}^-$ and $H_{s}^+$ for $s_0\le s\le s_1$ we also get $\lambda_s^+=-\lambda_{s-1}^-$.
		
		By possibly multiplying Equation~\eqref{eq:dependence} by $-1$ we can assume that $\lambda_{s_0}^+\ge 0$.
		The above arguments then yield $\lambda_{s_0}^-\le 0$ and $\lambda_{s_0+1}^+\ge 0$.
		We thus iteratively get that $\lambda_{s}^+\ge 0$ for all $s_0\le s\le s_1$.
		Considering the coordinate corresponding to the element $n+1$ in Equation~\eqref{eq:dependence} immediately yields $\lambda_{s}^+ = 0$ for all $s_0\le s\le s_1$.
		The above arguments on the alternation of the sign then also imply $\lambda_{s}^- = 0$ for all $s_0\le s\le s_1$.
		Therefore all $\lambda$-coefficents must actually be $0$ and the hyperplanes in $\A_{i}\setminus \A_{i-1}$ are independent.
		
		The second step is to show that the hyperplane $H_I\in \mathcal{A}_i\setminus \A_{i-1}$ satisfy condition (3) of Theorem~\ref{thm:MAT}.
		By Lemma~\ref{lem:triples} it suffices to show that every hyperplane $H_I\in \mathcal{A}_i\setminus \A_{i-1}$ is involved in exactly $i-1$ circuit-triples with the hyperplanes in $\A_{i-1}$.
		Since all sets corresponding to hyperplanes in $\A_{i-1}$ are strictly smaller than the ones for $\A_i\setminus \A_{i-1}$,  we need to show that there are exactly $i-1$ pairs of hyperplanes $H_A,H_B$ in $\A_{i-1}$ such that $A\dot{\cup} B=I$ for a $H_I\in \A_i\setminus \A_{i-1}$.
		
		First, let $H_I$ be a hyperplane in $\A_i\setminus \A_{i-1}$ with $n+1\not\in I$.
		Hence, $I=\{j,j+1,\dots,j+i-1\}$ for some $1\le j\le n-i+1$.
		Then we can decompose $I$ in exactly $i-1$ ways via $I=\{j,\dots,j+\ell\}\cup \{j+\ell+1,\dots,j+i-1\}$ for $\ell=0,\dots,i-2$.
		
		Now let  $H_I$ be a hyperplane with $n+1\in I$.
		Hence, $I=\{j,j+1,\dots,j+i-2\}\cup \{n+1\}$ for some $1\le j\le n-i+2$ such that $j\le k\le j+i-2$.
		One way to decompose $I$ is by placing $n+1$ as a singleton.
		Subsequently there are $i-2$ more ways of splitting up the chain $\{j,j+1,\dots,j+i-2\}$ as in the first part with the only difference that we add the element $n+1$ to the part that contains the element $k$.
		Thus every hyperplane $H_I$ is involved in exactly $i-1$ circuit-triples with the hyperplanes in $\A_{i-1}$.
		Therefore, $\A_i$ is MAT free.
		
		The statement on the exponents is obtained by counting connected subgraphs and taking the transposed partition of
		$\pi_{n,k}=(|\A_i\setminus \A_{i-1}| \mid i=1,\ldots,n+1)$.
		Since the number of $H_I$ in $\A_i\setminus \A_{i-1}$ with $n+1\in I$ is
		$$|\{j\in\{1,\ldots,n-i+2\} \mid j\le k \le j+i-2\}| = \min \{i-1,k,n+2-i,n+1-k\}, $$
		we have
		$\pi_{n,k}=(n+1)\cup (n-i+1+\min \{i-1,k,n+2-i,n+1-k\} \mid i=2,\ldots,n+1)$, or
		\begin{equation}\label{eq:partition1}
		   \pi_{n,k}=(n+1,\underbrace{n,\ldots,n}_{k},n-1,n-2,\ldots,2k-1,2k-3,\ldots,3,1) 
		\end{equation}
		for $k\le (n+1)/2$, and $\pi_{n,k}=\pi_{n,(n+1-k)}$ for $k>(n+1)/2$. Thus for $k\le (n+1)/2$, the dual of $\pi_{n,k}$ is
		\begin{equation}\label{eq:partition2}
            \pi_{n,k}^T = (n+1,\underbrace{n,n,n-1,n-1,\ldots,n-k+2,n-k+2}_{k-1\text{ pairs}},n-k+1,\ldots,k+1,1),
        \end{equation}
		and of course $\pi_{n,k}^T=\pi_{n,(n+1-k)}^T$ for $k>(n+1)/2$.
		Note that in Equations \eqref{eq:partition1} and \eqref{eq:partition2} the subsequences $(n-1,n-2,\ldots,2k-1)$ and $(n-k+1,\ldots,k+1)$ should be read as the empty sequence when $k=(n+1)/2$.
	\end{proof}
	
	\begin{cor}\label{cor:charpoly_nk}
		The characteristic polynomial of the arrangement $\A_{A_{n,k}}$ is 
        \[
		\chi(\A_{A_{n,k}},t)= (t-1)\prod_{i=k+1}^{n+1}(t-i)\prod_{i=n-k+2}^{n}(t-i).
		\]
	\end{cor}

\begin{example}\label{example:MAT}
The \csa to the Dynkin diagram of type $E_6$ has normal vectors:
\begin{center}
(0,0,0,0,0,1),(0,0,0,0,1,0),(0,0,0,0,1,1),(0,0,0,1,0,0),(0,0,0,1,0,1),\\
(0,0,0,1,1,1),(0,0,1,0,0,0),(0,0,1,0,0,1),(0,0,1,0,1,1),(0,0,1,1,0,1),\\
(0,0,1,1,1,1),(0,1,0,0,0,0),(0,1,0,0,1,0),(0,1,0,0,1,1),(0,1,0,1,1,1),\\
(0,1,1,0,1,1),(0,1,1,1,1,1),(1,0,0,0,0,0),(1,0,0,1,0,0),(1,0,0,1,0,1),\\
(1,0,0,1,1,1),(1,0,1,1,0,1),(1,0,1,1,1,1),(1,1,0,1,1,1),(1,1,1,1,1,1)
\end{center}
This is $\A_{A_{n,k}}$ for $n+1 = 6$, $k=3$.
We have $6,5,5,5,3,1$ of these roots with coordinate sum $1,2,3,4,5,6$ respectively. We can visualize this in a partition:
\begin{center}
\begingroup%
  \makeatletter%
  \providecommand\rotatebox[2]{#2}%
  \newcommand*\fsize{\dimexpr\f@size pt\relax}%
  \newcommand*\lineheight[1]{\fontsize{\fsize}{#1\fsize}\selectfont}%
  \ifx\svgwidth\undefined%
    \setlength{\unitlength}{182.8962091bp}%
    \ifx\svgscale\undefined%
      \relax%
    \else%
      \setlength{\unitlength}{\unitlength * \real{\svgscale}}%
    \fi%
  \else%
    \setlength{\unitlength}{\svgwidth}%
  \fi%
  \global\let\svgwidth\undefined%
  \global\let\svgscale\undefined%
  \makeatother%
  \begin{picture}(1,0.96777248)%
    \lineheight{1}%
    \setlength\tabcolsep{0pt}%
    \put(0,0){\includegraphics[width=\unitlength,page=1]{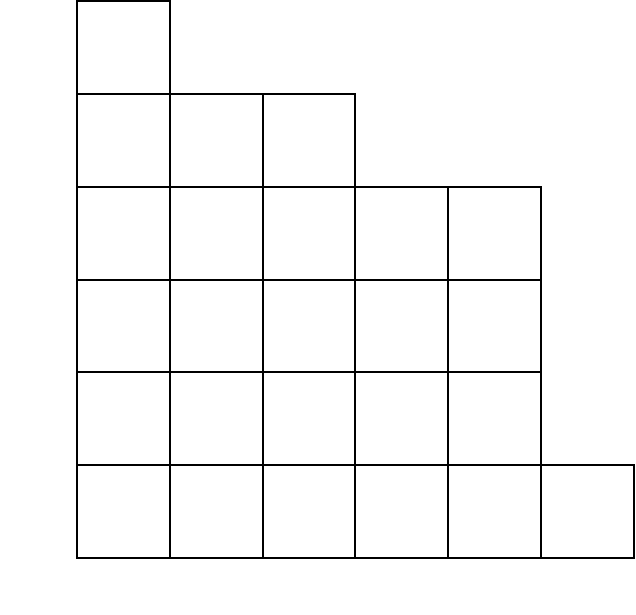}}%
    \put(-0.00270148,0.15385213){\color[rgb]{0,0,0}\makebox(0,0)[lt]{\lineheight{1.25}\smash{\begin{tabular}[t]{l}6\end{tabular}}}}%
    \put(-0.00298486,0.30350619){\color[rgb]{0,0,0}\makebox(0,0)[lt]{\lineheight{1.25}\smash{\begin{tabular}[t]{l}5\end{tabular}}}}%
    \put(-0.00298486,0.45265026){\color[rgb]{0,0,0}\makebox(0,0)[lt]{\lineheight{1.25}\smash{\begin{tabular}[t]{l}5\end{tabular}}}}%
    \put(-0.00298486,0.60179433){\color[rgb]{0,0,0}\makebox(0,0)[lt]{\lineheight{1.25}\smash{\begin{tabular}[t]{l}5\end{tabular}}}}%
    \put(-0.00294708,0.75042829){\color[rgb]{0,0,0}\makebox(0,0)[lt]{\lineheight{1.25}\smash{\begin{tabular}[t]{l}3\end{tabular}}}}%
    \put(-0.00425059,0.90008241){\color[rgb]{0,0,0}\makebox(0,0)[lt]{\lineheight{1.25}\smash{\begin{tabular}[t]{l}1\end{tabular}}}}%
    \put(0.17060946,0.00054795){\color[rgb]{0,0,0}\makebox(0,0)[lt]{\lineheight{1.25}\smash{\begin{tabular}[t]{l}6\end{tabular}}}}%
    \put(0.32022769,0.00054795){\color[rgb]{0,0,0}\makebox(0,0)[lt]{\lineheight{1.25}\smash{\begin{tabular}[t]{l}5\end{tabular}}}}%
    \put(0.46951525,0.00054795){\color[rgb]{0,0,0}\makebox(0,0)[lt]{\lineheight{1.25}\smash{\begin{tabular}[t]{l}5\end{tabular}}}}%
    \put(0.61874623,-0){\color[rgb]{0,0,0}\makebox(0,0)[lt]{\lineheight{1.25}\smash{\begin{tabular}[t]{l}4\end{tabular}}}}%
    \put(0.76803397,-0){\color[rgb]{0,0,0}\makebox(0,0)[lt]{\lineheight{1.25}\smash{\begin{tabular}[t]{l}4\end{tabular}}}}%
    \put(0.91683974,-0){\color[rgb]{0,0,0}\makebox(0,0)[lt]{\lineheight{1.25}\smash{\begin{tabular}[t]{l}1\end{tabular}}}}%
  \end{picture}%
\endgroup%
\end{center}
The dual partition gives the roots
$(1,4,5,6,4,5) = (1,k+1,\dots,n+1,n-k+2,\dots,n)$
of the characteristic polynomial.
\end{example}
	
	\section{Path-with-triangle graphs}\label{sec:triangle_path}
	In this section we consider the \csas associated with the \df{path-with-triangle graphs} $\Delta_{n,k}$ for $1<k<n$ which are graphs on the vertices $\left[n+1\right]$ consisting of a path from $1$ to $n$ together with two additional edges: one connecting $k$ and $n+1$ and one connecting $k+1$ and $n+1$.
	For instance, Figure~\ref{fig:path_triangle} depicts the graph $\Delta_{6,3}$ in this family.
	
	\begin{figure}[ht]
		\begin{tikzpicture}[scale=0.5,mynode/.style={font=\color{#1}\sffamily,circle,draw=black,inner sep=1pt,minimum size=0.5cm}]
		\node[mynode=black] (v1) at (0,0) {$1$};
		\node[mynode=black] (v2) at (3,0) {$2$};
		\node[mynode=black] (v3) at (6,0) {$3$};
		\node[mynode=black] (v4) at (9,0) {$4$};
		\node[mynode=black] (v5) at (12,0) {$5$};
		\node[mynode=black] (v6) at (15,0) {$6$};
		
		\node[mynode=black] (v7) at (7.5,-2) {$7$};
		
		\draw[gray,very thick] (v1)--(v2)--(v3)--(v4)--(v5)--(v6) (v3)--(v7) --(v4);
		\end{tikzpicture}
		\caption{The path-with-triangle graph $\Delta_{6,3}$.}
		\label{fig:path_triangle}
	\end{figure}
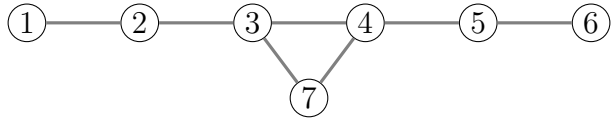
	
	\begin{proposition}\label{prop:path_triangle_restriction}
		The \csa $\A_{\Delta_{n,k}}$ is the restriction of the arrangement $\A_{A_{n+1,k+1}}$ to the hyperplane $x_{k+1}=0$.
	\end{proposition}
	\begin{proof}
		The restriction of the arrangement $\A_{A_{n+1,k+1}}$ to the coordinate hyperplane $\ker x_{k+1}$ is the arrangement defined by the same equations after removing every occurrence of $x_{k+1}$.
		After relabeling the variables $x_i\mapsto x_{i-1}$ for $k+2\le i\le n+2$ one obtains exactly the defining equations of the \csa $\A_{\Delta_{n,k}}$.
	\end{proof}
	
	\begin{remark}
		In general, the restriction of a \csa $\A_G$ to the hyperplane $\ker x_v$ corresponding to a vertex $v$ of $G$ yields the \csa $\A_{H}$ where $H$ is the graph that arises from $G$ by removing the vertex $v$ and adding the edges between all neighbors of $v$ in $G$.
	\end{remark}
	
	\begin{proposition}\label{prop:almost_triangle_automorphism}
		There exists a linear automorphism $\psi$ of $(\QQ^n)^*$ that induces a bijection of the hyperplanes of the \csa $\A_{A_{n,k}}$ which maps the hyperplane $\ker \sum_{j=1}^{n+1} x_j$ to the hyperplane $\ker x_k$.
	\end{proposition}
	\begin{proof}
		Consider the linear map $\psi:(\QQ^n)^* \to (\QQ^{n})^*$ given by
		\[
		x_i \mapsto 
		\begin{cases}
		-x_{k-i} & \mbox{ if }i <k,\\
		\sum_{j=1}^{n+1} x_j & \mbox{ if }i =k,\\
		-x_{n+1-(i-k)} & \mbox{ if }k<i\le n,\\
		-x_{n+1} &\mbox{ if }i =n+1.
		\end{cases}
		\]
		This map $\psi$ is a linear automorphism of $(\QQ^n)^*$ and clearly maps the hyperplane $\ker \sum_{j=1}^{n+1} x_j$ to the hyperplane $\ker x_k$.
		On the remaining hyperplanes of $\A_{A_{n,k}}$ it has the following effect:
		\begin{description}
			\item[$\ker\sum_{j=a}^b x_j \mbox{ for } 1\le a \le b < k$] The map $\psi$ sends this hyperplane to $\ker\sum_{j=k-b}^{k-a} x_j$ which is a hyperplane of $\A_{A_{n,k}}$. The case of $k<a\le b\le n$ works analogously.
			\item[$\ker\sum_{j=a}^b x_j+\lambda x_{n+1} \mbox{ for } 1\le a \le k \le b \le n$ and $\lambda\in \{0,1\}$] The map $\psi$ sends this hyperplane to
			\[
			\ker \sum_{j=k-a+1}^{n+1-(b-k)} x_j+(1-\lambda)x_{n+1},
			\]
			which is a hyperplane of $\A_{A_{n,k}}$.
			\item[$\ker x_{n+1}$] The map $\psi$ keeps this hyperplane invariant by construction.
		\end{description}
		Hence, the map $\psi$ induces a bijection of the hyperplanes of $\A_{A_{n,k}}$ as claimed.
	\end{proof}
	
	\begin{theorem}\label{thm:path_triangle}
		The \csa $\A_{\Delta_{n,k}}$ is free with exponents $(1,k+2,\dots,n+1,n-k+2,\dots,n+1)$.
		Hence the characteristic polynomial of $\A_{\Delta_{n,k}}$ is
		\[
		\chi(\A_{\Delta_{n,k}},t)= (t-1)\prod_{i=k+2}^{n+1}(t-i)\prod_{i=n-k+2}^{n+1}(t-i).
		\]
	\end{theorem}
	\begin{proof}
		By Theorem~\ref{thm:mat_almost_path} the \csa $\A_{A_{n+1,k+1}}$ is (MAT) free with exponents $(1,k+2,\dots,n+2,n-k+2,\dots,n+1)$.
		In the last step of the MAT filtration of $\A_{A_{n+1,k+1}}$ we add the hyperplane $\ker \sum_{i=1}^{n+2}x_i$.
		Therefore, the arrangement $\A_{A_{n+1,k+1}}\setminus \{ \ker \sum_{i=1}^{n+2}x_i\}$ is free with exponents $(1,k+2,\dots,n+1,n+1,n-k+2,\dots,n+1)$.
		
		Proposition~\ref{prop:almost_triangle_automorphism} now yields that also the arrangement $\A_{A_{n+1,k+1}}\setminus \{ \ker x_{k+1}\}$ is free with exponents $(1,k+2,\dots,n+1,n+1,n-k+2,\dots,n+1)$.
		The addition-deletion Theorem~\ref{thm:addition_deletion} therefore implies that the restriction of $\A_{A_{n+1,k+1}}$ to the hyperplane $\ker x_{k+1}$ is free with exponents $(1,k+2,\dots,n+1,n-k+2,\dots,n+1)$.
		By Proposition~\ref{prop:path_triangle_restriction} this is exactly the arrangement $\A_{\Delta_{n,k}}$ which finishes the proof.
	\end{proof}
	
	\section{Converse direction of the proof}\label{sec:converse}
	In this section we complete the proof of Theorem~\ref{thm:main} by proving that the \csas of all other graphs are not free.
	We start by discussing local properties of arrangements.
	\subsection{Local properties}\label{subsec:loc_prop}
	Recall that given an element $X\in L(\A)$ in the intersection lattice of an arrangement $\A$, the \df{localization} $\A_X$ of $\A$ at $X$ is defined as $\A_X\coloneqq \{H\in \A\mid X\subseteq H\}$.
	
	\begin{definition}
	We call a property $P$ of an arrangement \df{local} if it is persevered under localizations, that is if $\A$ has property $P$ then $\A_X$ for all $X\in L(\A)$.
	\end{definition}

	The properties of being free, simplicial, factored or supersolvable are all local properties, cf.~\cite[Thm. 4.37]{OT92},~\cite[Prop 3.2]{Sta72},~\cite[Cor. 3.90]{OT92}, and \cite[Lemma 2.27]{OT92}, respectively.
	We proceed by proving two lemmas on how local properties are inherited by certain graph theoretic constructions.
	
	\begin{lemma}\label{lem:induced_subgraph}
		The $G=(N,E)$ be a graph and $P$ be a local arrangement property.
		Assume the \csa $\A_G$ has property $P$.
		Let $S\subseteq N$ be a subset of nodes and $G\left[S\right]$ be the induced subgraph on $S$.
		Then $\A_{G\left[S\right]}$ also has property $P$.
	\end{lemma}
	\begin{proof}
		Consider the element $X_S \coloneqq \cap_{i\in S}H_i,$ in the intersection lattice of $\A_G$
		where $H_{i}$ is the hyperplane $\ker x_i$.
		Since the property $P$ is by assumption local, the arrangement $(\A_G)_{X_S}$ also has property $P$.
		Therefore it suffices to prove $(\A_G)_{X_S}=\A_{G\left[S\right]}$ where we regard $\A_{G\left[S\right]}$ in the larger ambient space $\QQ^{|N|}$.
		
		Indeed, for a hyperplane $H_B\in \A_G$ with $B\subseteq N$ we have $H_B\in (\A_G)_{X_S}$ if and only if $X_S \subseteq H_B$.
		By construction of $X_S$ this is equivalent to the condition that every point $p\in\QQ^{|N|}$ with $p_i=0$ for all $i\in S$ lies in $H_B$.
		This is in turn equivalent to $B\subseteq S$ and hence to $H_B\in \A_{G\left[S\right]}$ as $H_B$ was assumed to be an element of $\A_G$, that is the subgraph $G\left[B\right]$ is connected.		
	\end{proof}

    \begin{definition}
        Let $G=(N,E)$ be a graph and $e=\{i,j\}\in E$ an edge.
        We denote by $G/e$ the graph with \df{contracted edge} $e$, i.e.\ the graph in which the vertices $i,j$ are identified to one vertex $k$, each edge to $i$ or $j$ becomes an edge to $k$, and multiplicities of edges or loops are removed ($G/e$ is a simple graph).
    \end{definition}

	\begin{lemma}\label{lem:edge_contraction}
		The $G=(N,E)$ be a graph and $P$ be a local arrangement property.
		Assume the \csa $\A_G$ has property $P$.
		Let $e=\{i,j\}\in E$ be an edge of $G$ and denote by $G/e$ the graph with the contracted edge $e$.
		Then $\A_{G/e}$ also has property~$P$.
	\end{lemma}
	\begin{proof}
		Consider the element
		\[
		X_e \coloneqq \bigcap_{B\subseteq N,\: |B\cap\{i,j\}|\neq 1}H_B,
		\]
		in the intersection lattice of $\A_G$.
		Again since the property $P$ is by assumption local, the arrangement $(\A_G)_{X_e}$ also has property $P$.
		
		Let $G'=(N',E')$ be the graph that arises from $G$ by contracting the edge $e$.
		We can assume that $N'=N\setminus\{j\}$ as contracting the edge $e=\{i,j\}$ identifies the two nodes $i$ and $j$.
		We define a linear embedding $\phi : (\QQ^{|N'|})^*\hookrightarrow (\QQ^{|N|})^*$ by setting $\phi(x_k)\coloneqq x_k$ for $k\neq i$ and $\phi(x_i)\coloneqq x_i+x_j$.
		We claim that $\phi$ maps the hyperplanes of $\A_{G'}$ bijectively to the hyperplanes in $(\A_G)_{X_e}$ which proves by the above discussion that $\A_{G'}$ also has property $P$ as desired.
		
		Let $H'_{B'}$ with $B'\subseteq N'$ be a hyperplane in $\A_{G'}$. To prove that $\phi$ maps $H'_{B'}$ to a hyperplane in $(\A_G)_{X_e}$ we consider the following two cases.
		\begin{description}
			\item[Case 1] Suppose $i\notin B'$. By construction, $\phi$ maps $H'_{B'}$ to the hyperplane $H_B\in \A_G$ with $B=B'$ as we have $G'[B']=G[B]$ in this case.
			Thus, we also have $B\cap \{i,j\}=\emptyset$ and therefore $X_e\subseteq H_B$ by definition of $X_e$.
			Hence, $H_B\in (\A_G)_{X_e}$.
			\item[Case 2] Suppose $i\in B'$.  In this case, $\phi$ maps $H'_{B'}$ to the hyperplane $H_B$ with $B=B'\cup\{j\}$ by construction of $\phi$.
			The hyperplane $H_B$ is in $\A_G$ as $G'[B']$ is connected by assumption which implies that $G[B]$ is connected too.
			Lastly, as $\{i,j\} \subseteq B$ we again obtain $X_e\subseteq H_B$ and thus $H_B\in (\A_G)_{X_e}$.
		\end{description}
		To show that this map is bijective it suffices to prove it is surjective as the set of hyperplanes is finite.
		Let $H_B$ be a hyperplane in $(\A_G)_{X_e}$.
		By construction of $X_e$ this means that $|B\cap \{i,j\}|\neq 1$.
		The above discussion therefore shows that $\phi$ maps the hyperplane $H'_{B\cap N'}$ to the hyperplane~$H_B$ which completes the proof.
	\end{proof}
	
	\subsection{Classification of free \csas}
	In this subsection we classify the free \csas.
	Together with Proposition~\ref{prop:cycle} and Theorems~\ref{thm:mat_almost_path} and~\ref{thm:path_triangle} this concludes the proof of Theorem~\ref{thm:main}.
	We start with the following proposition.

	\begin{figure}[ht!]
\begingroup%
  \makeatletter%
  \providecommand\color[2][]{%
    \errmessage{(Inkscape) Color is used for the text in Inkscape, but the package 'color.sty' is not loaded}%
    \renewcommand\color[2][]{}%
  }%
  \providecommand\transparent[1]{%
    \errmessage{(Inkscape) Transparency is used (non-zero) for the text in Inkscape, but the package 'transparent.sty' is not loaded}%
    \renewcommand\transparent[1]{}%
  }%
  \providecommand\rotatebox[2]{#2}%
  \newcommand*\fsize{\dimexpr\f@size pt\relax}%
  \newcommand*\lineheight[1]{\fontsize{\fsize}{#1\fsize}\selectfont}%
  \ifx\svgwidth\undefined%
    \setlength{\unitlength}{373.19945595bp}%
    \ifx\svgscale\undefined%
      \relax%
    \else%
      \setlength{\unitlength}{\unitlength * \real{\svgscale}}%
    \fi%
  \else%
    \setlength{\unitlength}{\svgwidth}%
  \fi%
  \global\let\svgwidth\undefined%
  \global\let\svgscale\undefined%
  \makeatother%
  \begin{picture}(1,0.56056855)%
    \lineheight{1}%
    \setlength\tabcolsep{0pt}%
    \put(0,0){\includegraphics[width=\unitlength,page=1]{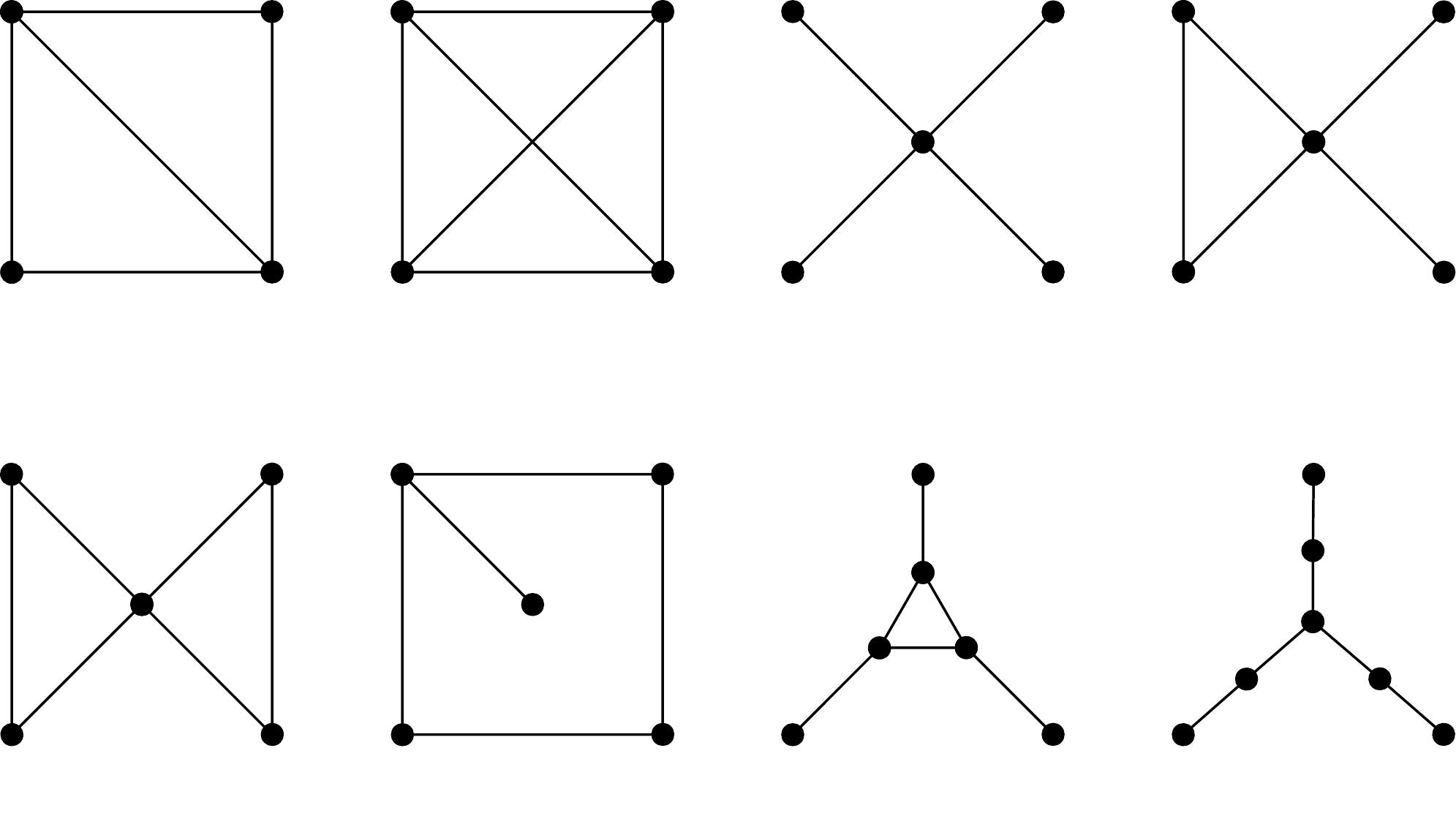}}%
    \put(0.00297192,0.31601094){\color[rgb]{0,0,0}\makebox(0,0)[lt]{\lineheight{1.25}\smash{\begin{tabular}[t]{l}Graph $G_1$\end{tabular}}}}%
    \put(0.27012925,0.31601094){\color[rgb]{0,0,0}\makebox(0,0)[lt]{\lineheight{1.25}\smash{\begin{tabular}[t]{l}Graph $G_2$\end{tabular}}}}%
    \put(0.53240701,0.31601094){\color[rgb]{0,0,0}\makebox(0,0)[lt]{\lineheight{1.25}\smash{\begin{tabular}[t]{l}Graph $G_3$\end{tabular}}}}%
    \put(0.80566381,0.31601094){\color[rgb]{0,0,0}\makebox(0,0)[lt]{\lineheight{1.25}\smash{\begin{tabular}[t]{l}Graph $G_4$\end{tabular}}}}%
    \put(-0.00068777,0.00493734){\color[rgb]{0,0,0}\makebox(0,0)[lt]{\lineheight{1.25}\smash{\begin{tabular}[t]{l}Graph $G_5$\end{tabular}}}}%
    \put(0.26646957,0.00493734){\color[rgb]{0,0,0}\makebox(0,0)[lt]{\lineheight{1.25}\smash{\begin{tabular}[t]{l}Graph $G_6$\end{tabular}}}}%
    \put(0.5287473,0.00493734){\color[rgb]{0,0,0}\makebox(0,0)[lt]{\lineheight{1.25}\smash{\begin{tabular}[t]{l}Graph $G_7$\end{tabular}}}}%
    \put(0.80200411,0.00493734){\color[rgb]{0,0,0}\makebox(0,0)[lt]{\lineheight{1.25}\smash{\begin{tabular}[t]{l}Graph $G_8$\end{tabular}}}}%
  \end{picture}%
\endgroup%
		\caption{Obstruction graphs for freeness with up to $7$ nodes.}
		\label{fig:8graphs}
	\end{figure}
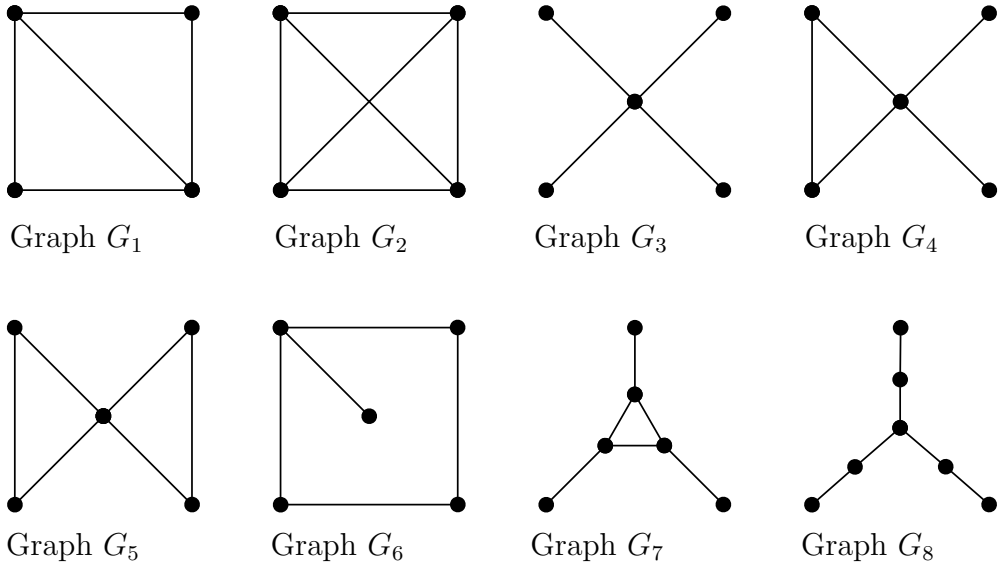

	\begin{proposition}\label{prop:8_graphs}
    The characteristic polynomials of the \csas of the graphs in Figure \ref{fig:8graphs} are:
    \begin{center}
    \setlength{\tabcolsep}{2pt}
	\renewcommand{\arraystretch}{1.3}
    \begin{tabular}{L L@{}}
    \toprule
    \text{Graph } $G\quad$ & \chi(\A_G) \\
    \midrule \rowcolor{lightgray}
     G_1 & (t - 1)(t - 4)(t^2 - 9t + 21) \\
     G_2 & (t - 1)(t - 4)(t^2 - 10t + 26) \\\rowcolor{lightgray}
     G_3 & (t - 1)(t - 4)(t - 5)(t^2 - 10t + 29) \\
     G_4 & (t - 1)(t - 5)(t^3 - 15t^2 + 78t - 138) \\\rowcolor{lightgray}
     G_5 & (t - 1)(t - 5)^2(t^2 - 11t + 33) \\
     G_6 & (t - 1)(t - 5)^2(t^2 - 10t + 26) \\\rowcolor{lightgray}
     G_7 & (t - 1)(t - 5)(t^4 - 23t^3 + 200t^2 - 784t + 1180) \\
     G_8 & (t - 1)(t - 5)(t - 7)(t^4 - 23t^3 + 200t^2 - 784t + 1188)\\
    \bottomrule
    \end{tabular}
    \end{center}
    \medskip
    Note that the factors of the polynomials on the right are irreducible over $\QQ$.
	\end{proposition}
	\begin{proof}
	These polynomials are obtained for example via a tedious computation using deletion-restriction or with a computer using for instance the \texttt{julia} package described in~\cite{BEK_resonance}.
	\end{proof}

	\begin{proposition}\label{prop:small_graphs}
		Let $G$ be a connected graph on at most $7$ nodes that contains at least one of the graphs
		$G_1,G_3,G_6,G_7,G_8$ of Figure~\ref{fig:8graphs} as a subgraph.
		Then $\A_G$ is not free.
	\end{proposition}
	\begin{proof}
        Let $G$ be a connected graph on at most $7$ nodes that contains at least a graph $H\in\{G_1,G_3,G_6,G_7,G_8\}$ from Figure \ref{fig:8graphs} as a subgraph.
        Since localizations of free arrangements are free, by Terao's factorization theorem~\ref{thm:factorization} and Lemmas \ref{lem:induced_subgraph}, \ref{lem:edge_contraction} it suffices to show that $G$ contains an induced subgraph $H'$ with possibly contracted edges such that the characteristic polynomial of $\A_{H'}$ has non integral roots.
        
        But adding edges to $G_1,G_3,G_6,G_7,G_8$ always yields graphs which contain a graph of Figure~\ref{fig:8graphs} as an induced graph after possibly contracting edges. Instead of going through all cases, we demonstrate the argument in an example case: consider adding an edge to $G_6$ from the middle vertex to the vertex at the bottom left. Then contracting the edge on the right yields the graph $G_1$.
        
        Thus the subgraph of $H$ in $G$ contains an induced graph $H'$ with possibly contracted edges as in  Figure~\ref{fig:8graphs} and we are finished.
	\end{proof}
	
	This proposition together with the above graph-manipulation-lemmas yields a classification for all graphs:
	\begin{theorem}\label{thm:classification}
		Let $\A_G$ be a free \csa of a connected graph $G$.
		Then $G$ is a path, a cycle, an almost path or a path-with-triangle-graph.
	\end{theorem}
	\begin{proof}
		We proceed in several steps:
		\begin{enumerate}
			\item Assume $G$ has a node $v$ with at least four neighbors $v_1,\dots,v_4$
			and set $S=\{v,v_1,\dots,v_4\}$.
			By Lemma~\ref{lem:induced_subgraph} the \csa $\A_{G\left[S\right]}$ is also free, but this contradicts Proposition~\ref{prop:small_graphs} as $G\left[S\right]$ has $G_3$ in Figure~\ref{fig:8graphs} as a subgraph.
			Therefore all nodes in $G$ are of degree at most three.
			\item Assume $G$ has at least two cycles $C_1,C_2$ of length at least three that share at least one edge $\{v_1,v_2\}$.
			Let $w_i$ be a node on $C_i$ for $1\le i\le 2$ that is different from $v_1,v_2$ and that does not lie on the other cycle.
			These nodes exists as the two cycles are different and of length at least three.
			Contracting all edges on both cycles that involve nodes different from $v_1,v_2,w_1,w_2$ and then taking the induced subgraph on these four nodes yields a graph that has the graph $G_1$ in Figure~\ref{fig:8graphs} as a subgraph, where $\{v_1,v_2\}$ is the middle edge.
			The Lemmas~\ref{lem:induced_subgraph} and~\ref{lem:edge_contraction} yield that the associated \csa is free which contradicts Proposition~\ref{prop:small_graphs}.
			\item Assume $G$ has at least two cycles of length at least three that don't share an edge.
			In this case, one can obtain a graph as an induced subgraph after suitable edge contractions that contains the graph $G_3$ in Figure~\ref{fig:8graphs} as a subgraph.
			This is a contradiction to Proposition~\ref{prop:small_graphs} by a similar argument as above.
			Therefore, $G$ can have at most one cycle.
			\item Assume $G$ has a cycle of length at least four.
			Using similar arguments as before applied to the graph $G_6$ from Figure~\ref{fig:8graphs} now yields that the graph $G$ must be actually equal to a cycle graph which is one of the cases in our classification.
			So from now assume that $G$ has at most one cycle and that cycle has length three.
			\item Now let $v_1,v_2,v_3$ be the unique cycle in $G$. Assume that all nodes $v_1,v_2,v_3$ are of degree three. Taking the induced subgraph of $v_1,v_2,v_3$ together with their other neighbor each yields the graph $G_7$ from Figure~\ref{fig:8graphs} since all neighbors are distinct as $v_1,v_2,v_3$ is the only cycle in $G$.
			Using as above Lemma~\ref{lem:induced_subgraph} this contradicts Proposition~\ref{prop:small_graphs}.
			Therefore we can assume that the node $v_3$ has degree two.
			
			If $G$ had a node $v$ of degree three different from $v_1$ and $v_2$, then one could obtain a graph with a node of degree at least four by contracting all edges between $v$ and $v_1$ or $v$ and $v_2$.
			This again contradicts Proposition~\ref{prop:small_graphs} as the contracted graph has an induced subgraph that contains $G_3$ from Figure~\ref{fig:8graphs}.
			Therefore, all other nodes have degree at most two and as $v_1,v_2,v_3$ is its only cycle, $G$ must be a path-with-triangle-graph.
			\item Now assume $G$ is a tree.
			If $G$ had two nodes $v_1,v_2$ of degree three, one could again contract all edges between $v_1,v_2$ and obtain a graph with a node of degree four.
			This is impossible as shown in the previous step.
			Therefore there exists at most one node of degree three.
			If all degrees are two, $G$ is a path and we're done.
			
			So say $v$ is the unique node of degree three.
			Therefore $G$ is a graph consisting of three disjoint paths that are glued together in the node $v$.
			The obstruction graph $G_8$ in Figure~\ref{fig:8graphs} ensures that at least one of these paths is in fact just an edge.
			Therefore, $G$ is an almost path in our sense.
		\end{enumerate}
		We  have now exhausted all cases and found that $\A_G$ can only be free if $G$ is a path, cycle, almost path or path-with-triangle-graph.
		We previously showed that $\A_G$ is indeed free in these four cases.
	\end{proof}

	\section{Simplicial arrangements}\label{sec:simplicial}
	
	\subsection{Simplicial arrangements}
	\begin{definition}\label{def:simp} An arrangement $\A$ in $\QQ^n$ is called \df{simplicial} if the complement
		$\QQ^n\setminus \bigcup_{H\in \A}H$ decomposes into open simplicial cones.
	\end{definition}
	Simpliciality is a local property (see Subsection \ref{subsec:loc_prop}), hence we may investigate simplicial \csas using the Lemmas \ref{lem:induced_subgraph} and \ref{lem:edge_contraction}:
	\begin{theorem}\label{thm:ss_p}
		Let $G$ be a connected graph.
		Then the \csa $\A_G$ is simplicial if and only if $G$ is a triangle or a path graph.
	\end{theorem}
	\begin{figure}[ht!]
		\centering
		\includegraphics[width=.8\linewidth]{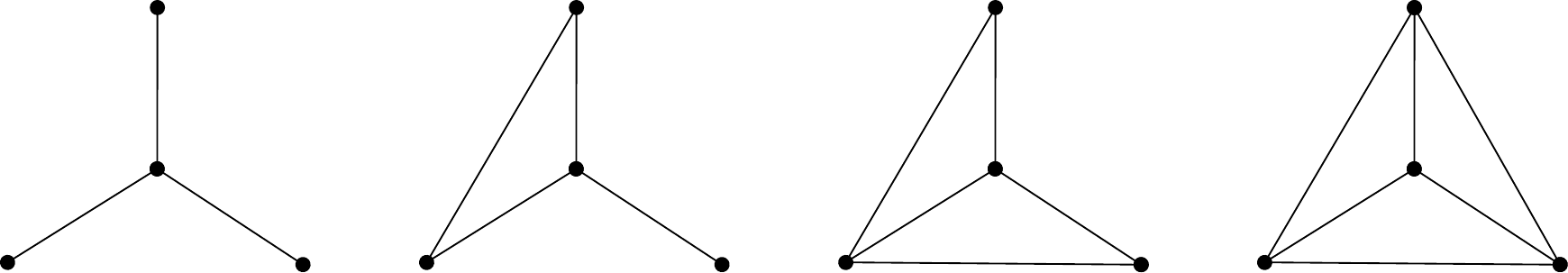}
		\caption{Possible localizations to a branching.}
		\label{fig:branch}
	\end{figure}
	\begin{proof}
		If $G$ is a path graph, then $\A_G$ is the Weyl arrangement of type $A$ which is known to be simplicial.
		If $G$ is a triangle, then $\A_G$ is the arrangement to the root system of the Weyl groupoid of rank three with $7$ positive roots; this is also simplicial (see for example \cite{p-C10}).
		
		Now assume that $G$ has more than three vertices. If $G$ contains some branching with at least three branches, then one of the graphs in Figure \ref{fig:branch} is an induced graph $H$.
		Consider for $H$ the first graph in Figure \ref{fig:branch}, its \csa $\A$ has characteristic polynomial $(t-1)(t-3)^2(t-4)$ (Corollary \ref{cor:charpoly_nk}). By \cite[Cor.\ 2.4]{p-CG-13}, $\A$ is simplicial if and only if
		\begin{equation}\label{eq:simplicial}
		n \chi(\A,-1)+2\sum_{H\in \A} \chi(\A^H,-1) = 0.
		\end{equation}
		The arrangements $\A^H$ are isomorphic to $\A_{P_3}$ if $H$ is a hyperplane corresponding to one of the three leaves in $H$; all other restrictions $\A^H$ are isomorphic to $\A_{C_3}$.
		Since $\chi(\A_{P_3})=(t-1)(t-2)(t-3)$ and $\chi(\A_{C_3})=(t-1)(t-3)^2$, the left side of Equation \eqref{eq:simplicial} gets
		\[ 4\cdot 160 + 2(3\cdot (-24) + 8\cdot (-32)) = -16 \ne 0, \]
		and hence $\A$ is not simplicial.

        We can check in the same way that none of the graphs in Figure \ref{fig:branch} produces a simplicial \csa,
		thus we get a contradiction to Lemma \ref{lem:induced_subgraph}.
		Hence $G$ is a path or a cycle. Assume that it is a cycle of length greater than three. Then after some edge contractions we obtain a cycle of length four. But this is not a simplicial \csa (use Equation \eqref{eq:simplicial}), thus this case is excluded by Lemma \ref{lem:edge_contraction}.
	\end{proof}
	
	\section{Factored and supersolvable arrangements}\label{sec:factored}
    This section contains a classification of factored and supersolvable \csas.
    We begin by defining these two notions.
	\begin{definition}\cite[Def. 2.66]{OT92}\label{def:factored}
		Let $\pi = (\pi_1,\ldots,\pi_s)$ be a partition of $\A$.
		The partition $\pi$ is called \df{independent}, if for any choice $H_i\in \pi_i$ for $1\le i\le s$, the resulting $s$ hyperplanes are linearly independent.
		
		 Let $X \in L(\A)$. The \df{induced partition} $\pi_X$ of $\A_X$ is given by the non-empty blocks of the form $\pi_i\cap \A_X$.		
		The partition $\pi$ is \df{nice} for $\A$ or a \df{factorization} of $\A$ if
		
		\begin{enumerate}
			\item $\pi$ is independent, and
			\item for each $X \in L(\A)\setminus \{V\}$, the induced partition $\pi_X$ admits a block which is a singleton.
		\end{enumerate}
		If $\A$ admits a factorization, then we also say that $\A$ is \df{factored} or \df{nice}.
	\end{definition}
	
	\begin{example}[{\cite[Exp.~2.5]{Ter92}}]\label{ex:triangle_factored}
		The \csa of the triangle graph is factored with the partition 
		\begin{align*}
			\pi_1&=\{H_{\{1,3\}} \},\\
			\pi_2&=\{H_{\{3\}},H_{\{1\}},H_{\{1,2\}} \},\\
			\pi_3&=\{H_{\{2\}},H_{\{2,3\}},H_{\{1,2,3\}} \}.\qedhere
		\end{align*}
	\end{example}
	
	\begin{definition}\label{def:ss}
	    Let $\A$ be an arrangement of rank $r$.
	    \begin{enumerate}
	        \item An $X\in L(\A)$ is \df{modular} if for every $Y\in L(\A)$ it holds that $X+Y\in L(\A)$.
	        \item The arrangement $\A$ is \df{supersolvable} if there exists a chain of modular elements $X_0 < X_1<\dots <X_r$ with $X_i\in L(\A)$ for all $0\le i \le r$.
	    \end{enumerate}
	\end{definition}
	
	The key facts about factored arrangements that we will use are the following:
	\begin{proposition}[{\cite[Cor.~3.88]{OT92}}]\label{prop:factored_chi}
		Let $\A$ be an arrangement with a nice partition $(\pi_1,\dots,\pi_s)$. 
		Then
		\[
		\chi(\A)=\prod_{i=1}^s(t-|\pi_i|).
		\]
	\end{proposition}
	\begin{proposition}[{\cite[Prop.~2.67]{OT92}}]\label{prop:factored_ss}
		Let $\A$ be a supersolvable arrangement with a maximal chain of modular elements $X_0<X_1<\dots < X_r$. Then the partition $(\pi_1,\dots,\pi_r)$ with $\pi_i = \A_{X_i}\setminus A_{X_{i-1}}$ for $1\le i \le r$ is a factorization of $\A$.
	\end{proposition}

	If an arrangement $\A$ has a factorization $\pi=(\pi_1,\dots,\pi_s)$ then for $X\in L(\A)$ the nonzero blocks of the induced partition $\pi_X$ form a factorization of $\A_X$, see also the proof of Cor.~3.90 in~\cite{OT92}.
	Therefore, being factored is a local property in the sense as defined in Section~\ref{sec:converse}.

	To prove that some arrangements are not factored we introduce a notion of conflicting sets.
	This definition gives a set-theoretic description when the localization $\A_{H_A\cap H_B}$ only contains the two hyperplanes $H_A$ and $H_B$ in $\A$. This will be explained in Lemma~\ref{lem:factored}.
	\begin{definition}\label{def:conflicting}
		Let $\A$ be a $0/1$-arrangement in $F^n$ and let $H_A,H_B\in \A$ be two different hyperplanes in $\A$ with $A,B\subseteq \left[ n \right]$.
		We call $A$ and $B$ with $A\neq B$ \df{conflicting} if one of these two conditions hold:
		\begin{enumerate}
			\item $A\cap B \neq \emptyset$, $A\not \subseteq B$, and $B\not \subseteq A$.
			\item Suppose $A\cap B=\emptyset$, $A\subset B$, or $B \subset A$.
			Then the hyperplane $H_{A\Delta B}$ is not in $\A$, where $A\Delta B$ is the symmetric difference of the sets $A,B$. 
		\end{enumerate}
		These conditions ensure that there is no circuit-triple involving $H_A$ and $H_B$ in $\A$.
	\end{definition}
	\begin{lemma}\label{lem:factored}
			Suppose $\A$ is a $0/1$-arrangement with a nice partition $\pi$ and let $H_A,H_B\in \A$ with conflicting sets $A,B$.
			Then the hyperplanes $H_A,H_B$ are in different blocks of the partition $\pi$.
	\end{lemma}
	\begin{proof}
		Set $X=H_A\cap H_B$.
		As the sets $A,B$ are conflicting, Lemma~\ref{lem:triples} implies $\A_X=\{ H_A,H_B\}$.
		Then by definition the induced partition $\pi_X$ must also contain a singleton block which means that $H_A$ and $H_B$ must be in different blocks of the partition $\pi$.
	\end{proof}
	
	\begin{proposition}\label{prop:not_factored}
		The \csas of the following graphs are not factored:
		\begin{enumerate}
			\item The almost path graph $A_{3,2}$ (a star graph with central node $2$ and outer nodes $1,3,4$).
			\item The square graph $C_4$ (cycle on $4$ nodes).
		\end{enumerate}
	\end{proposition}
	\begin{proof}
		For (1) consider these two families of subsets of $\{1,2,3,4\}$:
		\begin{align*}
			\mathcal{F}_1 =& \{\{2\},\{1,2,3\},\{1,2,4\},\{2,3,4\}\},\\
			\mathcal{F}_2 =& \{\{1,2\},\{2,3\},\{2,4\},\{1,2,3,4\}\}.
		\end{align*}
		Note that the hyperplanes corresponding to these sets are all in the \csa $\A_{A_{3,2}}$ and these two families are disjoint.
		Furthermore within both families, the sets are pairwise conflicting in this arrangement.
		For instance in the family $\mathcal{F}_1$ the set $\{2\}$ is conflicting with all other sets in this family by condition (2) in Definition~\ref{def:conflicting} as this arrangement does not contain the hyperplanes $H_{\{1,3\}}$, $H_{\{1,4\}}$, and $H_{\{3,4\}}$.
		The latter three sets in $\mathcal{F}_1$ are conflicting by condition (1) in this definition.
		
		Now suppose that $\A_{A_{3,2}}$ has a nice partition $\pi$.
		By Lemma~\ref{lem:factored} the four hyperplanes of the family $\mathcal{F}_1$ must all lie in different blocks of $\pi$.
		The same holds for the four hyperplanes of the family $\mathcal{F}_2$.
		But this is impossible as the partition $\pi$ must contain a singleton block by Proposition~\ref{prop:factored_chi} as $1$ is a root of $\chi(\A_{A_{3,2}})$.
		Therefore $\A_{A_{3,2}}$ is not factored.
		
		Now for (2) consider the arrangement $\A_{C_4}$ and suppose it has a nice partition $\pi=(\pi_1,\pi_2,\pi_3,\pi_4)$. 
		It has the family of pairwise conflicting sets $$\{ \{1,2,3\},\{1,2,4\},\{1,3,4\},\{2,3,4\}\}.$$
		Therefore by Lemma~\ref{lem:factored} the corresponding hyperplanes must lie in different parts of $\pi$.
		Without loss of generality we can assume that
		\[
			H_{\{1,2,3 \}}\in \pi_1,\quad H_{\{1,2,4 \}}\in \pi_2,\quad H_{\{1,3,4 \}}\in \pi_3, \mbox{ and } H_{\{2,3,4 \}}\in \pi_4.
		\]
		By Proposition~\ref{prop:cycle} we have $\chi(\A_{C_4})=(t-1)(t-4)^3$.
		Thus we can by Proposition~\ref{prop:factored_chi} assume that $|\pi_1|=1$ and $|\pi_i|=4$ for $i=2,3, 4$ and hence $\pi_1=\{H_{\{1,2,3 \}}\}$.
		
		The families of sets $\{\{1\},\{2,3\}\}$, $\{\{1,2\},\{3\}\}$, and $\{\{4\},\{1,2,3,4\}\}$ each form a circuit-triple with the set $\{1,2,3\}$.
		Thus by definition of a nice partition the two corresponding hyperplanes of each of these families must lie in the same block of $\pi$.
		The set $\{1,2\}$ is conflicting with $\{1,3,4\}$ and $\{2,3,4\}$.
		Analogously, the set $\{2,3\}$ is conflicting with $\{1,2,4\}$ and $\{1,3,4\}$.
		Thus we must have
		\[
		\{ H_{\{1,2,4 \}},H_{\{1,2 \}},H_{\{3 \}}\}\subset \pi_2,\, \{ H_{\{1,3,4 \}},H_{\{4 \}},H_{\{1,2,3,4 \}}\}\subset  \pi_3, \{ H_{\{2,3,4 \}},H_{\{1 \}},H_{\{2,3 \}}\}\subset  \pi_4.
		\]
		Now the set $\{1,4\}$ is conflicting with the sets $\{1,2\}$ and $\{2,3,4\}$.
		Thus $H_{\{1,4\}}\in \pi_3$.
		Analogously the set $\{3,4\}$ is conflicting with the sets $\{1,2,,4\}$ and $\{2,3\}$ and therefore we must also have $H_{\{3,4\}}\in \pi_3$.
		As this would be the fifth hyperplane in $\pi_3$ we have a contradiction to the fact $|\pi_3|=4$ which means that $\A_{C_4}$ is not factored.
	\end{proof}
	\begin{figure}[hb]
	\begin{subfigure}{0.35\linewidth}
		\begin{tikzpicture}[scale=0.45,mynode/.style={font=\color{#1}\sffamily,circle,draw=black,inner sep=1pt,minimum size=0.5cm}]
		\node[mynode=black] (v1) at (-6,0) {$1$};
		\node[mynode=black] (v2) at (-3,0) {$2$};
		\node[mynode=black] (v3) at (0,0) {$3$};
		\node[mynode=black] (v4) at (3,0) {$4$};
		\node[mynode=black] (v5) at (-4.5,-3) {$5$};
		\node (v6) at (5,0) {};
		
		\draw[gray,very thick] (v1)--(v2)--(v3)--(v4) (v1) --(v5) --(v2);
		\end{tikzpicture}
		\end{subfigure}
		\begin{subfigure}[b]{0.6\linewidth}	
		\begin{center}
		\setlength{\tabcolsep}{2pt}
		\renewcommand{\arraystretch}{1.3}
        \begin{tabular}{*{2}{l}@{}}
        \toprule
        	Blocks $\quad$& Hyperplanes $H_A$ for the sets $A\subseteq \left[5\right]$\\
        	\midrule \rowcolor{lightgray}
        	$\pi^{(4)}_1$ & $\{1,5\}$  \\ 
        	$\pi^{(4)}_2$ & $\{5\},\{1\},\{1,2\},\{1,2,3\},\{1,2,3,4\}$  \\\rowcolor{lightgray}
        	$\pi^{(4)}_3$ & $\{2\},\{2,5\},\{1,2,5\}$  \\
        	$\pi^{(4)}_4$ & $\{3\},\{2,3\},\{2,3,5\},\{1,2,3,5\}$   \\\rowcolor{lightgray}
        	$\pi^{(4)}_5$ & $\{4\},\{3,4\},\{2,3,4\},\{2,3,4,5\},\{1,2,3,4,5\}$   \\
        	\bottomrule
        \end{tabular}
        \end{center}
		\end{subfigure}
		\caption{The path-with-triangle-graph $\Delta_{4,1}$ and the blocks of the nice partition $\pi^{(4)}$ as defined in the proof of Proposition~\ref{prop:factored}.}
		\label{fig:A41}
	\end{figure}
	\begin{proposition}\label{prop:factored}
		The \csa $A_{\Delta_{n,1}}$ of the path-with-triangle graph $\Delta_{n,1}$ for $n \ge 2$ is factored.
		These graphs are paths with an extra triangle at the two left-most nodes.
	\end{proposition}
	\begin{proof}
		Consider the following partition of the hyperplanes of $\A_{\Delta_{n,1}}$:
		\begin{align*}
		    \pi_1^{(n)} &\coloneqq \{ H_{\{1,n+1\}} \},\\
		    \pi_2^{(n)} &\coloneqq \{ H_{\{n+1\}},H_{\{1\}},H_{\{1,2\}},\dots,H_{\{1,2,\dots,n\}} \},\\
		    \pi_i^{(n)} &\coloneqq \{ H_X \mid X\cap \{i,\dots,n\}=\emptyset \} \setminus \bigcup_{j=1}^{i-1}\pi_j\quad \mbox{for }3\le i \le n+1.
		\end{align*}
		Figure~\ref{fig:A41} depicts the graph $\Delta_{4,1}$ together with the associated partition $\pi^{(4)}$.
		We now show by induction on $n\ge 2$ that the partition $\pi^{(n)}=(\pi_1,\dots,\pi_{n+1})$ is nice.
		Example~\ref{ex:triangle_factored} states that the partition $\pi^{(2)}$ is nice.
		
		Note that in general $\pi^{(n-1)}$ arises from $\pi^{(n)}$ by removing the last block, removing the element $n-1$ from all sets defining the hyperplanes and mapping $n$ to $n-1$.
		We will frequently use this fact implicitly in the following arguments when applying the induction hypothesis.
		
		So now consider $\pi^{(n)}$ for some fixed $n>2$.
		We first show that $\pi^{(n)}$ is independent.
		Let $H_i\in \pi_i^{(n)}$ be a collection of hyperplanes for $1\le i \le n+1$ which we will show to be linearly independent.
		The sets defining the hyperplanes in $\pi_{n+1}^{(n)}$ all involve the element $n$ and the only hyperplane whose set contains $n$ in the other blocks is the hyperplane $H_{\{1,2,\dots,n\}}$ in $\pi_2^{(n)}$.
		We now distinghuish the following three cases:
		\begin{description}
		\item[Case 1] Suppose $H_2\neq H_{\{1,2,\dots,n\}}$.
		Then the collection $(H_1,\dots,H_{n})$ is independent by induction and thus the collection $(H_1,\dots,H_{n+1})$ is independent as $H_{n+1}$ is the only hyperplane whose set contains $n$.
		\item[Case 2] Suppose $H_2=H_{\{1,2,\dots,n\}}$ and $H_{n+1}\neq H_{\{2,\dots,n+1\}}$.
		In this case, the intersection of $H_2$ and $H_{n+1}$ is contained in some hyperplane in $\pi_2^{(n)}$ by construction, say $H_2'$ with $H_2'\neq H_2$.
		The collection $(H_1,H_2',H_3\dots,H_{n+1})$ is then again independent by induction as $H_{n+1}$ is the only hyperplane that involves the element $n$.
		Thus by the above observation, also the collection $(H_1,H_2,H_3\dots,H_{n+1})$ is independent.
		\item[Case 3] Suppose $H_2=H_{\{1,2,\dots,n\}}$ and $H_{n+1}= H_{\{2,\dots,n+1\}}$.
		In this case, $H_{\{1\}}$ contains the intersection of $H_1=H_{\{1,n+1\}}$, $H_2=H_{\{1,2,\dots,n\}}$ and $H_{n+1}=H_{\{2,\dots,n+1\}}$.
		Therefore as above the collection $(H_1,H_{\{1\}},H_3\dots,H_{n+1})$ is independent by induction which again implies that the collection $(H_1,H_2,H_3\dots,H_{n+1})$ is independent.
		\end{description}
		
		As a second step we prove the singleton block condition.
		Let $X$ be any element in $L(\A_{\Delta_{n,1}})$.
		Since there is no circuit-triple of three hyperplanes whose corresponding sets all contain $n$, we can express $X$ as
		\[
		    X=H_A\cap \bigcap_{i=1}^s H_{B_i},
		\]
		for sets $A,B_i\subseteq \left[n+1\right]$ and some $s\ge 1$ with the condition that $n\not \in B_i$ for all $1\le i \le s$.
		Set $Y\coloneqq \bigcap_{i=1}^s H_{B_i}$.
		Following the above comment, the subspace $Y$ is also in $L(\A_{\Delta_{n-1,1}})$.
		Therefore by induction, the induced partition $\pi^{(n)}_Y$ has a singleton block, say $(\pi^{(n)}_Y)_k$ for some $k$.
		We again need to distinguish these cases:
		\begin{description}
		\item[Case 1] Suppose $k\neq 2$.
		Then also the induced partition $\pi^{(n)}_X$ must have a singleton block as the extra hyperplanes in this larger localized arrangement are all added in the second or last block since the circuit-triples involving these new hyperplanes contain at most one hyperplane in ${\A_{\Delta_{n,1}}}_Y$.
		\item[Case 2] Suppose $k=2$ and $A\neq \left[n\right]$.
		In this case the hyperplanes of ${\A_{\Delta_{n,1}}}_X\setminus {\A_{\Delta_{n,1}}}_Y$ are all contained in the block $\pi^{(n)}_{n+1}$.
		Thus, $(\pi^{(n)}_{X})_2$ is a singleton block.
		\item[Case 3] Suppose $k=2$ and $A= \left[n\right]$.
		In this case the intersection of $H_A$ and the unique hyperplane in $(\pi^{(n)}_Y)_2$ is contained in exactly one hyperplane in $\pi^{(n)}_{n+1}$.
		Therefore, $(\pi^{(n)}_{X})_{n+1}$ is a singleton block in $\pi^{(n)}_{X}$.\qedhere
		\end{description}
	\end{proof}

	\begin{theorem}\label{thm:factored}
		Let $G$ be a connected graph.
		The \csa $A_G$ is factored if and only if $G$ is a path graph or a path-with-triangle graph $\Delta_{n,1}$.
	\end{theorem}
	\begin{proof}
		The \csa of a path graph is the braid arrangement and therefore known to be supersolvable.
		By Proposition~\ref{prop:factored_ss} these arrangements are also factored.
		Proposition~\ref{prop:factored} shows that the \csas of the graphs $\Delta_{n,1}$ for  $n \ge 2$ are factored.
		
		For the converse let $G$ be a connected graph with at least $4$ nodes and assume the \csa $\A_G$ is factored.
		Suppose the graph has a node of degree at least $4$.
		Then it contains an induced graph $H$ with $G_3$ in Figure~\ref{fig:8graphs} as a subgraph.
		By Proposition~\ref{prop:small_graphs} the characteristic polynomial of the \csa of $H$ does not factor over the integers which implies by Proposition~\ref{prop:factored_chi} that $\A_G$ is not factored as being factored is a local property.
		Therefore all nodes in $G$ have degree at most $3$.
		
		Now suppose $G$ contains a cycle of length at least $4$.
		After contracting some edges of $G$ there is an induced subgraph which is either the cycle graph $C_4$, the graph $G_1$ of Figure~\ref{fig:8graphs} or the complete graph $K_4$.
		The \csa of the cycle graph $C_4$ is not factored by Proposition~\ref{prop:not_factored}~(2).
		The characteristic polynomial of the \csa of the remaining graphs does not factor over the integers by Proposition~\ref{prop:small_graphs}.
		Therefore as above $\A_G$ can't be factored.
		An analogous argument using the graphs $G_1$ and $G_5$ in Figure~\ref{fig:8graphs} shows that $G$ can have at most one cycle of length $3$.
		
		Now assume that $G$ contains exactly one cycle of length $3$.
		Then $G$ contains as an induced subgraph either the graph $G_4$ in Figure~\ref{fig:8graphs}, or the path-with-triangle graphs $\Delta_{4,2}$ or $\Delta_{n,1}$ for $n\ge 2$.
		The characteristic polynomial of the graph $G_4$ does not factor by Proposition~\ref{prop:8_graphs}.
		Contracting the edge $\{2,3\}$ in the graph $\Delta_{4,2}$ yields the almost-path-graph $A_{3,2}$ whose \csa is not factored by Proposition~\ref{prop:not_factored}~(1).
		Thus the only possibility for factored \csas are the ones corresponding to the graphs  $\Delta_{n,1}$ for $n\ge 2$.
		
		Lastly, assume $G$ is a tree.
		If $G$ has a node of degree $3$ then it has the graph $A_{3,2}$ as an induced subgraph whose \csa is not factored by Proposition~\ref{prop:not_factored}~(1).
		Hence, $G$ must be a path graph in this case which finishes the proof.		
	\end{proof}

	\begin{cor}\label{cor:ss}
		Let $G$ be a connected graph.
		The \csa $A_G$ is supersolvable if and only if $G$ is a path graph.
	\end{cor}
	\begin{proof}
		The \csa of a path graph is the braid arrangement which is supersolvable.
		For the converse note that by Proposition~\ref{prop:factored_ss} a supersolvable arrangement is factored.
		Therefore, Theorem~\ref{thm:factored} yields that $G$ can only be a triangle or a path-with-triangle graph $\Delta_{n,1}$ for some $n\ge 2$.
		It is known that the \csa of the triangle graph is not supersolvable, see for instance~\cite[p. 49]{OT92}.
		As being supersolvable is a local property and all graphs $\Delta_{n,1}$ can be contracted to the triangle graph this yields that the \csa of the graph $\Delta_{n,1}$ is not supersolvable for all $n\ge 2$.
	\end{proof}

\newcommand{\etalchar}[1]{$^{#1}$}

\end{document}